\documentclass[10pt,letterpaper]{amsart}
\usepackage{ amsmath,amsthm, amsbsy,amsfonts,amssymb, txfonts}
\usepackage{stmaryrd,mathrsfs,graphicx, amscd, tikz-cd, tikz, cancel}
\usepackage{tikz,enumitem}
\usetikzlibrary{matrix,arrows,decorations.pathmorphing}

\usepackage[active]{srcltx}
\allowdisplaybreaks
\numberwithin{equation}{section}

\usepackage[
	hypertexnames=false,
	hyperindex,
	pagebackref,
	pdftex,
	breaklinks=true,
	bookmarks=false,
	colorlinks,
	linkcolor=blue,
	citecolor=red,
	urlcolor=red,
]{hyperref}
\usepackage{hyperref}

\def\CC{{\mathbb C}}

\def\PP{{\mathbb P}}
\def\Pb{{\mathbb P}}
\def\QQ{{\mathbb Q}} 
\def\RR{{\mathbb R}}

\def\ZZ{{\mathbb Z}} 
\def\Zb{{\mathbb Z}} 

\def\Ab{\mathbf {A}}
\def\Eb{\mathbf {E}}
\def\Ib{\mathbf {I}}

\def\Ub{\mathbf{U}}

\def\ssm{\smallsetminus}

\def\aff{{\rm aff}}

\def\G{\Gamma}
\def\g{\gamma}

\def\hol{\mathrm{hol}}

\def\orb{{\rm orb}}

\def\Itwo{{\rm I_2}}

\def\bs{\backslash}
\newcommand{\eps}{\varepsilon}

\newcommand{\p}{\partial}

\def\Ecal{{\mathcal E}}

\def\Mcal{{\mathcal M}}
\def\Ocal{{\mathcal O}}

\def\Scal{{\mathfrak S}}

\def\Xcal{{\mathcal X}}

\def\Bscr{{\mathscr B}}  
\def\Cscr{{\mathscr C}}  
\def\Dscr{{\mathscr D}}  
\def\Escr{{\mathscr E}}  
\def\Fscr{{\mathscr F}}  
\def\Gscr{{\mathscr G}}

\def\Tscr{{\mathscr T}}
\def\Uscr{{\mathscr U}}

\def\la{\langle}
\def\ra{\rangle}

\def\pt{{\scriptscriptstyle\bullet}}

\newcommand\aut{\operatorname{Aut}}

\newcommand\Diff{\operatorname{Diff}}
\newcommand\discr{\operatorname{Discr}}

\newcommand\Homeo{\operatorname{Homeo}}
\newcommand\im{\operatorname{Im}}

\newcommand\Mod{\operatorname{Mod}}

\newcommand\MW{\operatorname{MW}}

\newcommand\GL{\operatorname{GL}}

\newcommand\SL{\operatorname{SL}}

\newcommand\Orth{\operatorname{O}}
\newcommand\SO{\operatorname{SO}}

\newcommand\var{\operatorname{var}}

\newcommand\Trans{\operatorname{Trans}}

\newcommand\rank{\operatorname{rank}}

\newtheorem{theorem}{Theorem}[section]
\newtheorem{lemma}[theorem]{Lemma}
\newtheorem{proposition}[theorem]{Proposition}
\newtheorem{corollary}[theorem]{Corollary}

\theoremstyle{remark}
\newtheorem{definition}[theorem]{Definition}
\newtheorem{example}[theorem]{Example}
\newtheorem{remark}[theorem]{Remark}
\newtheorem{remarks}[theorem]{Remarks}

\newtheorem{question}[theorem]{Question}

\title[The smooth Mordell-Weil group]{\vspace{-.6in}The smooth Mordell-Weil group and \\ mapping class groups of elliptic surfaces}
\author{Benson Farb} 
\address{Dept.\ of Mathematics, Univ.\ of Chicago\\
5734 S. University Ave.\\
Chicago, IL 60614 (USA)}
\email{bensonfarb@gmail.com}

\author{Eduard Looijenga}
\address{Dept.\ of Mathematics, Univ.\ of Chicago\\
5734 S. University Ave.\\
Chicago, IL 60614 (USA)}
\email{e.j.n.looijenga@uu.nl}

\thanks{The first author was supported in part by National Science Foundation Grant No. DMS-181772 and the Eckhardt Faculty Fund. The  second author was supported by the University of Chicago.}
\date{May 9, 2024}

\begin{document}
\
%\subjclass[2000]{81R10, 17B65}
%\keywords{}

\begin{abstract}
This is a paper in smooth $4$-manifold topology, inspired by the 
N\'{e}ron-Lang Theorem in number theory.  More precisely, we prove 
that a smooth version $\MW(\pi)$ of Mordell-Weil group  of an elliptic fibration $\pi:M\to\Pb^1$ is finitely generated.  We compute $\MW(\pi_d)$ 
explicitly for elliptic fibrations $\pi_d:M_d\to\Pb^1$, where $M_d$ 
is a simply-connected complex surfaces $M_d$ of arithmetic genus 
$d\geq 1$ and all fibers of $\pi_d$ are nodal.

We prove in this case that the fibered structure is unique up  topological isotopy.  By combining  this with a result of Donaldson, we obtain the following remarkable consequence: any diffeomorphism of $M_d$ with $d\geq 3$ is topologically isotopic to a diffeomorphism taking fibers to fibers.
\end{abstract}

\maketitle

\vspace{-.3in}
\tableofcontents

\section{Introduction} 
This is a paper in smooth $4$-manifold topology, inspired by the 
Lang-N\'{e}ron Theorem in number theory.  More precisely, we prove a smooth version of the Mordell-Weil Theorem and apply it to the 
`unipotent radical' case of a Thurston-type classification of mapping classes of simply-connected elliptic surfaces.  We also give methods to compute a number of related subgroups of mapping class groups of 
these surfaces.   Applications include Nielsen realization theorems for simply-connected elliptic surfaces. In order to state our main results, we first briefly  review the corresponding chapter in the theory of complex surfaces.

\subsection{Elliptically fibered surfaces}\label{subsect:1.1} In this paper a  (holomorphic) {\em genus one fibration} of a smooth, compact complex surface $M$ is a holomorphic map $\pi:M\to\Pb^1$ whose general fiber is of genus one. We will always assume that 
\begin{enumerate}
\item $M$ is simply connected,
\item $\pi$ is {\em relatively minimal}:  no fiber contains a  copy of $\PP^1$ with  self-intersection $-1$, 
\item the fiber class of $\pi$  in $H_2(M)$ is primitive (= indivisible).
\end{enumerate}
 This implies among other things that $\pi$ has at least one singular fiber but has no multiple fibers, and  that all singular fibers are of Kodaira type. 

Such a fibration has only finitely many critical values, and each such 
value has a well-defined multiplicity $>0$ given by the Euler characteristic  of the fiber; the resulting divisor on $\PP^1$ is the \emph{discriminant $D_\pi$} of $\pi$. A point of $\Pb^1$ has  multiplicity $1$ in  $D_\pi$ precisely when the  fiber over it is a rational nodal curve; topologically this is a $2$-sphere with two distinct points identified. 
We say that the genus one fibration is \emph{generic} if all singular fibers are of that type, in other words, if the discriminant divisor $D_\pi$ is reduced. As the name suggests, every genus one fibration can, by an arbitrary small deformation, be  deformed into a generic one. An {\em elliptic fibration} is a genus one fibration equipped with a section of $\pi$.

The most important topological invariant of a genus one fibration $\pi : M\to \Pb^1$  is  the Euler characteristic of $M$; this is always a positive integer divisible by $12$  and the  quotient $d$  is called  \emph{of arithmetic genus} of  $M$. This number determines the place of $M$ in Kodaira's classification of complex surfaces.
When  $d=1$ the surface is rational  (of Kodaira dimension $-\infty$).  
Indeed, such a surface is obtained by blowing up the fixed-point set of a general pencil of cubic curves in $\Pb^2$ and comes with many sections (every exceptional divisor is one; see Example \ref{example:rationalelliptic} for more details). When  $d=2$ it is a K3 surface  (of Kodaira dimension $0$) and for $d>2$ its Kodaira dimension is $1$.  

Moishezon proved that the 
simply-connected, genus one fibrations of a fixed arithmetic genus $d$  form a single connected family. This  makes $d$ a complete smooth invariant, so that  we can think of such genus one  fibrations as structures having the same underlying closed oriented $4$-manifold $M_d$.  One way to obtain an elliptic fibration $\pi_d:M_d\to\Pb^1$ of arithmetic genus $d$ is by pulling back $\pi_1$  along a morphism $\Pb^1\to\Pb^1$ of degree  $d$ whose discriminant  is disjoint with that of $\pi_1$. From this it is easy to see that $H_2(M_d)$ has signature $(2d-1,10d-1)$.

\subsection{The holomorphic Mordell-Weil group}
As explained in more detail in  \S\ref{subsection:trans} below, each smooth fiber of a holomorphic genus one fibration $\pi:M\to\Pb^1$ inherits an affine structure, and the notion of `translation' of a fiber makes sense.   
The {\em holomorphic Mordell-Weil group} is the group 
\[
\MW_{\hol}(\pi):=\{f\in\aut(M): \text{$f$ acts on each smooth fiber by translation}\},
\]
where $\aut(M)$ denotes the group of biholomorphic automorphisms of $M$.  This group is evidently abelian.  The passage to the  associated Jacobian fibration (which replaces each smooth fiber by its Jacobian) does not change $\MW_{\hol}(\pi)$ but turns it into the Mordell-Weil group of an elliptic fibration (rather than of a genus one fibration) and in that case it is a well-studied, classical object (see e.g.\ the book of Sch\"utt--Shioda, \cite{shioda}): it is the group of rational points of an elliptic curve over the field of rational functions on $\Pb^1$.  The Mordell-Weil Theorem for function fields (proved by Lang-N\'eron) implies that $\MW_\hol(\pi)$ is finitely generated (\footnote{Recall that we are assuming that $M$ is simply connected; otherwise finite generation does not hold.}).

The group $\MW_\hol(\pi)$ acts faithfully on $H_2(M):=H_2(M;\Zb)$ and  fixes the class $e\in H_2(M)$ of a fiber of $\pi$. The class $e$ is isotropic for the intersection pairing and primitive by assumption (3).  This makes  $e^\perp/\Zb e$ a nondegenerate lattice.  The vectors in $e^\perp$  whose Poincar\'e duals have Hodge type $(1,1)$  determine  a \emph{negative-definite} primitive \emph{even} sublattice of $e^\perp$ which we shall denote by $e^\perp_{alg}$. This sublattice is preserved by $\MW_\hol(\pi)$. It is known that $\MW_\hol(\pi)$  acts faithfully  on  $e^\perp_{alg}$, which means that $\MW_\hol(\pi)$ embeds in the 
semi-direct product  
$(e^\perp_{alg}/\Zb e)\rtimes \Orth(e^\perp_{alg}/\Zb e)$.  Since $\Orth(e^\perp_{alg})$ is finite, it follows that the  rank of $\MW_\hol(\pi)$ is at most the rank of    $e^\perp_{alg}/\Zb e$. The latter  is  known to be at most $10d-2$  (Thm.~ 6.20 in \cite{shioda}).  

In order to make the connection of $\MW_\hol(\pi)$ with its classical interpretation, assume that $\pi$ admits a holomorphic section $\PP^1\to M$.  Then  $\MW_\hol(\pi)$ permutes such sections  simply transitively. This leads to an identification of   $\MW_\hol(\pi)$  with $e^\perp_{alg}/e^\perp_{triv}$, where  $e^\perp_{triv} \subset e^\perp_{alg}$ denotes the sublattice generated by the classes of irreducible  components of fibers of $\pi$. 
If we divide out by torsion, then $\MW_\hol(\pi)$ comes (through the above characterization) with a negative definite quadratic form.  The  opposite of this form is  the so-called {\em height pairing} on $\MW_\hol(\pi)$ (which is indeed positive definite). If  all fibers are irreducible (which is the case of interest here), then $e^\perp_{triv}=\ZZ e$ and hence hence $\MW_\hol(\pi)\cong e^\perp_{alg}/\ZZ e$ is free abelian.

 In general the group $\MW_\hol(\pi)$ depends on $\pi$ and on the complex structure on $M$. 

\subsection{The smooth Mordell-Weil group}  For a genus one fibration 
$\pi:M\to\Pb^1$, the group
\[
\Trans(\pi):=\{f\in\Diff(M): \text{$f$ acts on each smooth fiber of $\pi$ by translation}\}.
\]
is abelian and infinite-dimensional. We define the {\em smooth Mordell-Weil group} $\MW(\pi)$ of $\pi$ to be its maximal discrete quotient
\[
\MW(\pi):=\pi_0(\Trans(\pi)).
\]
Here the implicit assumption is that $\pi$ is holomorphic, but the isomorphism type of $\Trans(\pi)$ and $\MW(\pi)$ only depends on the fiberwise diffeomorphism type of $\pi:M\to\Pb^1$.   As in the holomorphic case, the set of connected components of the space of smooth sections of $\pi$ is a torsor for $\MW(\pi)$, so that the elements of 
$\MW(\pi)$ are represented by fiberwise translations by the difference of two smooth sections.  See Figure \ref{figure:MWpic1}.

\begin{figure}[h]
\includegraphics[scale=0.11]{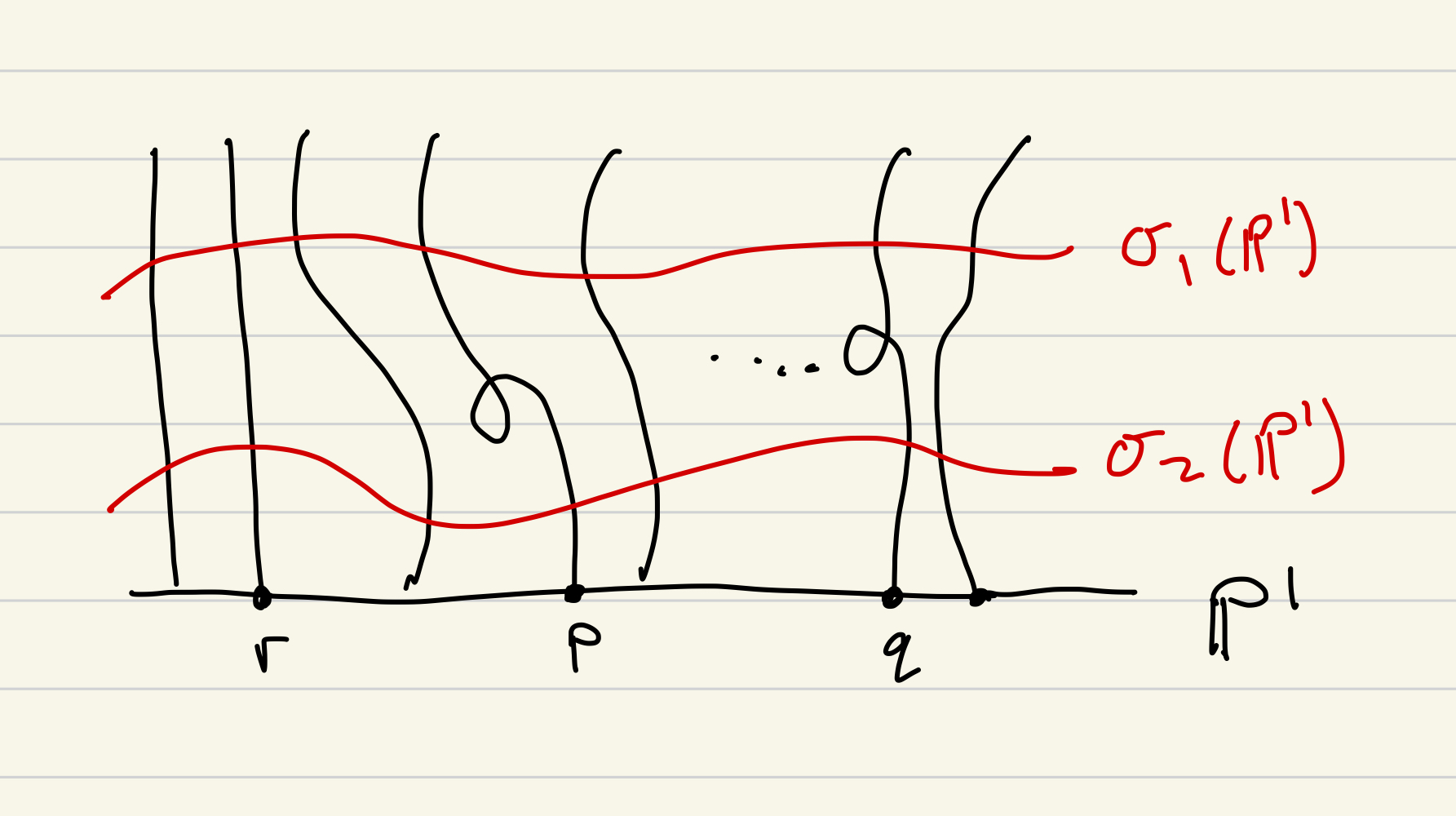}
\caption{\footnotesize A schematic of an elliptic fibration over $\Pb^1$.  The generic fiber, as over $r\in\Pb^1$ in the figure, is a smooth elliptic curve; the fibers over $p,q\in\Pb^1$ in the figure represent rational nodal curves. While such fibers are generic, a typical fiber of an elliptic fibration can be much more complicated.  Two smooth sections $\sigma_1,\sigma_2$ are indicated.  They determine a fiber-preserving diffeomorphism by translating each fiber $\pi^{-1}(b)$ by $z\mapsto z+\sigma_2(b)$, where addition is given in the unique group structure on $\pi^{-1}(b)$ with $\sigma_1(b)$ as identity.}
\label{figure:MWpic1}
\end{figure}

Our first result is a smooth analogue of the Lang-N\'eron theorem. (We emphasize that all the theorems stated here are  subject to our blanket assumptions  (1)--(3), as stated at the beginning of Subsection \ref{subsect:1.1}.)

\begin{theorem}[{\bf Smooth Mordell-Weil Theorem}]
\label{theorem:smoothMW4}
Let $\pi:M\to\Pb^1$ be any genus one fibration of a compact, complex surface $M$.  Then $\MW(\pi)$ is finitely-generated.
\end{theorem}

A more general and refined version of this result is 
given as Theorem \ref{thm:finitegen} below.  We will now focus 
on the generic case:  in what follows $\pi_d:M_d\to\Pb^1$ will always stand for a generic (holomorphic) genus one fibration of arithmetic genus $d$. We write $\Lambda_d$ for the lattice $H_2(M_d)$ and $e\in \Lambda_d$ for the fiber class. This class is primitive and isotropic. The following theorem then gives a complete description of the smooth Mordell-Weil group. 

\begin{theorem}[{\bf Computation of \boldmath$\MW(\pi_d)$}]
\label{theorem:computeMW}
Let $\pi_d:M_d\to\Pb^1$ be a generic genus one fibration of arithmetic genus $d\geq 1$. Then 
$\MW(\pi_d)$ acts faithfully on $\Lambda_d$.  Furthermore,  there is 
a natural  isomorphism of free abelian groups
\begin{equation}
\label{eq:computation6}
 \MW(\pi_d)\cong e^\perp/\ZZ e. 
 \end{equation}
The right hand side is an even {unimodular} lattice of signature $(2d-2, 10d-2)$ and hence isomorphic to $d\Eb_8(-1)\perp(2d-2)\Ub$, where $\Eb_8(-1)$ is the negative-definite $\Eb_8$ lattice  and $\Ub$ is the rank $2$ 
hyperbolic lattice (both are even unimodular).
\end{theorem}
Theorem \ref{theorem:computeMW} is proved at the end of Section \ref{section:MW}.

\begin{remark}[Holomorphic vs. smooth]
\label{remark:comparison1}
Theorem~\ref{theorem:computeMW} implies that
\[
\rank (\MW_{\rm hol}(\pi_d))\leq 10d-2 < 12d-4=\rank (\MW(\pi_d)) \ \ \text{for $d\geq 2$}
\]
highlighting the difference between the holomorphic and smooth 
categories.   In particular, the natural map $\MW_\hol(\pi_d)\to\MW(\pi_d)$ 
induced by the inclusion  $\MW_\hol(\pi_d)\to\Trans(\pi_d)$ cannot be onto when $d\ge 2$ (it  is an isomorphism for $d=1$).
\end{remark}
%\vspace{-.5cm}

\begin{remark}[Symplectic vs.\ smooth]
After posting a first version of this paper, it was pointed out to us that Hacking-Keating prove a symplectic version of Theorem \ref{theorem:computeMW}.  
More precisely, their version concerns the restriction of $\pi_d$ to an affine line in $\PP^1$ whose total space  they endow with  an  exact symplectic form such that the fibers are Lagrangian; see \S 4.1 of \cite{HK}. Their symplectomorphisms can be extended to all of $M_d$.  If $\Trans(\pi_d,\omega)$ denotes the fiberwise translations preserving a  specific symplectic form (given in \cite{HK}),  then they find an isomorphism $\pi_0(\Trans(\pi_d,\omega))\to e^\perp/\Zb e$.  In view of  the factorization
\[\pi_0(\Trans(\pi_d,\omega))\to\MW(\pi_d)\to e^\perp/\Zb e,\]
their result implies the surjectivity part of the isomorphism given in \eqref{eq:computation6}.
\end{remark}

\subsection{$\MW(\pi_d)$ as a mapping class group}  Our interest in the smooth Mordell-Weil group $\MW(\pi_d)$ arises from its close relationship with the {\em (smooth) mapping class group} of $M_d$, defined by
\[\Mod(M_d):=\pi_0(\Diff(M_d)).\]  

We claim that $f\in\Diff(M_d)$ preserves the orientation of $M_d$ and hence the intersection pairing on $M_d$.  If not then 
$\la \alpha\cup\beta, [M_d]\ra =\la f^*\alpha\cup f^*\beta, -[M_d]\ra $ for all $\alpha, \beta\in H^2(M_d)$, which would imply 
that the  Witt index of the intersection pairing is zero, i.e.,  $2d-1=10d-1$ and this is obviously never the case here.
So the map $\rho_d([f]):=f_*$ defines a representation $\rho_d:\Mod(M_d)\to \Orth(\Lambda_d)$.  The action of $\MW(\pi_d)$ on $\Lambda_d$ factors as a composition
\[
 \MW(\pi_d)\to \Mod(M_d)\to\Orth(\Lambda_d)
 \]
which is injective by Theorem \ref{theorem:computeMW}, so that we 
can regard $\MW(\pi_d)$ as a subgroup $\Mod(M_d)$.  It is not known whether the kernel of 
$\Mod(M_d)\to \Orth(\Lambda_d)$ is trivial or not.  Theorem \ref{theorem:computeMW} implies that $\MW(\pi_d)$ has trivial intersection with this kernel.  In fact a stronger result holds.

\begin{corollary}\label{cor:isomorphisms1}
For each $d\geq 1$ the map $\MW(\pi_d)\to \Orth(\Lambda_d)$ is injective, so that the following are equivalent for a fiberwise translation $F\in\Trans(\pi_d)\subset\Diff(M_d)$: 
\begin{enumerate}
\item $F$  acts as the identity in $\Lambda_d$,
\item $F$ is topologically isotopic to the identity,
\item $F$ is smoothly isotopic to the identity,
\item $F$ is smoothly fiberwise isotopic to the identity,
\item $F$ is smoothly isotopic to the identity via fiberwise translations in $M$.
\end{enumerate}
\end{corollary}

The equivalence of (1) and (2) in Corollary  \ref{cor:isomorphisms1} is 
given in \cite{GGHKP}; and  the implications (5)$\Rightarrow$(4)$\Rightarrow$(3)$\Rightarrow$(2)$\Rightarrow(1)$ are trivial.  Our contribution is Theorem \ref{theorem:computeMW}, which gives (1)$\Rightarrow$(5).
\\

The {\em Nielsen Realization Problem} asks whether a given homomorphism  $i:G\to\Mod(M_d)$ 
lifts to a homomorphism  $\tilde{i}: G\to \Diff(M_d)$; in this case we say that $G$ {\em is realized by} this lift.   

\begin{theorem}[{\bf Nielsen Realization, I}]
\label{theorem:Nielsen1}
For a generic genus one fibration $\pi:M_d\to \Pb^1$ of arithmetic genus $d\ge 1$,  the Nielsen Realization Problem is solvable for the injection $\MW(\pi_d)\to \Mod(M_d)$;  in fact, we can lift it to a homomorphism  $\MW(\pi_d)\to \Trans(\pi_d)$.\end{theorem}

We prove Theorem \ref{theorem:Nielsen1} at the end of Section \ref{section:MW}.  

\begin{remarks} We make a few comments on Theorem \ref{theorem:Nielsen1}.

(1)  The Nielsen Realization Problem has a negative answer for arbitrary subgroups of $\Mod(M_d)$; indeed there are finite subgroups $G\subset \Mod(M_2)$ that cannot be realized by a group of diffeomorphisms of $M_2$; see \cite{FL1}.

(2) Although each element of $\Trans(\pi_d)$ leaves invariant some complex structure, when $d\geq 2$ the whole group $\MW(\pi_d)$ 
cannot leave invariant any complex structure; see Remark \ref{remark:comparison1} above. 

(3) We can choose the lift in Theorem \ref{theorem:Nielsen1}  such that 
the action of $\MW(\pi_d)$ on $M_d$ via this  lift is free on the smooth part of each fiber with the topological closure of each orbit equalling that fiber. So then this lift determines $\pi$.

(4) When $\pi:M\to\Pb^1$ is not generic, that is,  when non-nodal singular fibers are allowed, then $\MW_\hol(\pi)$ (and hence $\MW(\pi)$) can contain torsion; see, e.g. \cite{shioda}.  It would be interesting to compute $\MW(\pi)$ in these cases.   
\end{remarks}

\subsection{Thurston-type classification and Eichler transformations}
\label{subsect:Thurstontype} 

The group $\MW(\pi_d)$ is a major case in a Thurston-type classification of elements of $\Mod(M_d)$, by which we mean finding for each class in $\Mod(M_d)$ an explicit representative that preserves a geometric structure and/or is optimal in some way.  

The real vector space $\RR\otimes \Lambda_d=H_2(M_d; \RR)$ admits a {\em spinor orientation}, which is an orientation on the tautological bundle of the Grassmannian of its (maximal) positive-definite $(2d-1)$-planes in $\RR\otimes \Lambda_d$.  Denote by $\Gamma_d, d\geq 2$ the index $2$ subgroup of $\Orth(\Lambda_d)$ consisting of those automorphisms preserving this spinor orientation and let $\Gamma_1:=\Orth(\Lambda_1)$.  Friedman-Morgan \cite{FM2} proved that the image of $\rho_d:\Mod(M_d)\to \Orth(\Lambda_d)$ lies in $\Gamma_d$.  

For $d=2$ (the case of a K3 surface), the authors  considered in \cite{FL1} the case when $f_*$ is of finite order.  Certain examples with $f_*$ of  infinite order and semisimple were explored in depth by McMullen \cite{Mc} from a 
dynamical perspective.  In this paper we consider the case when $f_*\in\Orth(\Lambda_m)$ is unipotent.   We show in Proposition \ref{prop:isotropicfixedvector} below that such an element must fix some primitive isotropic vector $e\in \Lambda_d$.  Remarkably, Friedman--Morgan  proved \cite{FM2} that when $d>2$  every element of $\Gamma_d$ has this 
property; see below.  We will show in Corollary \ref{cor:diffrep} that 
any unipotent element of $\Gamma_d$ can be represented by some $f\in\Diff(M_d)$ leaving invariant some elliptic fibration.  

Recall that $e\in \Lambda_d$ stands for the fiber class. Denote by $\Gamma_{d,e}$ the stabilizer in $\Orth(\Lambda_d)$ of $e$. Since $e$ primitive isotropic it is contained in $e^\perp$ and 
\[
\Lambda_d(e):= e^\perp/\ZZ e
\]
is a nondegenerate lattice.
The action of $\Gamma_{d,e}$  on $\Lambda_d(e)$ induces a representation $\Gamma_{d,e}\to\Orth(\Lambda_d(e))$ with image the index $2$ subgroup $\Gamma_d(e)$ fixing the spinor orientation.  This gives a (non-canonically split) short exact sequence
\[
0\to \Lambda_d(e)\to \Gamma_{d,e} \to \Gamma_d(e)\to 1.
\]
The elements of $\Lambda_d(e)$ are represented as \emph{Eichler transformations}:  every  $\tilde c\in e^\perp$ 
 defines an \emph{Eichler transformation} $E(e,\tilde c)\in \Gamma_{d,e}$ via
\[
E(e,\tilde c)(x):=x + (x\cdot e)\tilde c- (x\cdot \tilde c)e-\tfrac{1}{2}(\tilde c\cdot \tilde c)(x\cdot e)e.
\]
This transformation fixes $e$ and only depends on the image $c$ of $\tilde c$ in $\Lambda_d(e)$,  or better yet, on the $2$-vector $e\wedge \tilde c\in \wedge^2\Lambda_d$, so that we can write $E(e\wedge c)$ instead. In fact, 
any element of $\Gamma_d$ that fixes $e$ and acts trivially on $\Lambda_d(e)$ is of the form $E(e\wedge c)$ for a unique $c\in \Lambda_d(e)$.  

The group $\Lambda_d(e)<\Gamma_d$ is generated by Eichler transformations $E(e\wedge c)$ with $c^2=-2$.  The following theorem gives a ``best representative'' for these elements that is different than the one given in \ref{theorem:Nielsen1}.  Hacking-Keating (see \S 6.4 of \cite{HK}) have given such a construction in a symplectic setting, giving part (2) of the following result.  Recall that to an embedded $2$-sphere $C$ with self-intersection $-2$ there is associated {\em order $2$} Dehn twist $T(C)\in \Mod (M_d)$.

\begin{theorem}[{\bf Eichler representatives}]
\label{thm:smallsupport}
Let $\pi_d: M_d\to \Pb^1$  be a generic genus one fibration of arithmetic genus $d\geq 1$ and with fiber class $e\in \Lambda_d$. Given any 
$c\in e^\perp/\ZZ e$ with $c\cdot c=-2$, there exists:
\smallskip

\begin{enumerate}
\item an open disk $U_c\subset \Pb^1$ with two embedded $2$-spheres $C,C'\subset \pi^{-1}U_c$ representing resp.\ $c$ and $e-c$ such that $\pi_d$ has precisely two singular fibers over $p,q\in U_c$, and these define the same vanishing cycle in a smooth fiber over $U_c$. See Figure \ref{figure:smallsup1}.
\medskip

\item a diffeomorphism  $f_c:M_d\to M_d$ 
such that :
\smallskip
\begin{enumerate}
\item $f_c$ acts by fiberwise translations, has support in $\pi^{-1}U_c$  and induces in $\Lambda_d$ the Eichler transformation $E(e\wedge c)$  associated to $c$,
\item the Dehn twists associated to $C$ resp.\ $C'$ can be represented  by fiber-preserving diffeomorphisms 
$\tilde \tau(C)$ resp.\ $\tilde \tau(C')$ that lift a diffeomorphism of $\Pb^1$ supported on $U_c$ that interchanges $p$ and $q$
and represents a `simple braid',  such that  $\tilde \tau(C')\tilde \tau(C)^{-1}$ is isotopic to $f_c$ by a fiber-preserving isotopy.   See Figure \ref{figure:braid3}.
\end{enumerate}
\end{enumerate}

\end{theorem}

\begin{figure}[h]
\includegraphics[scale=0.11]{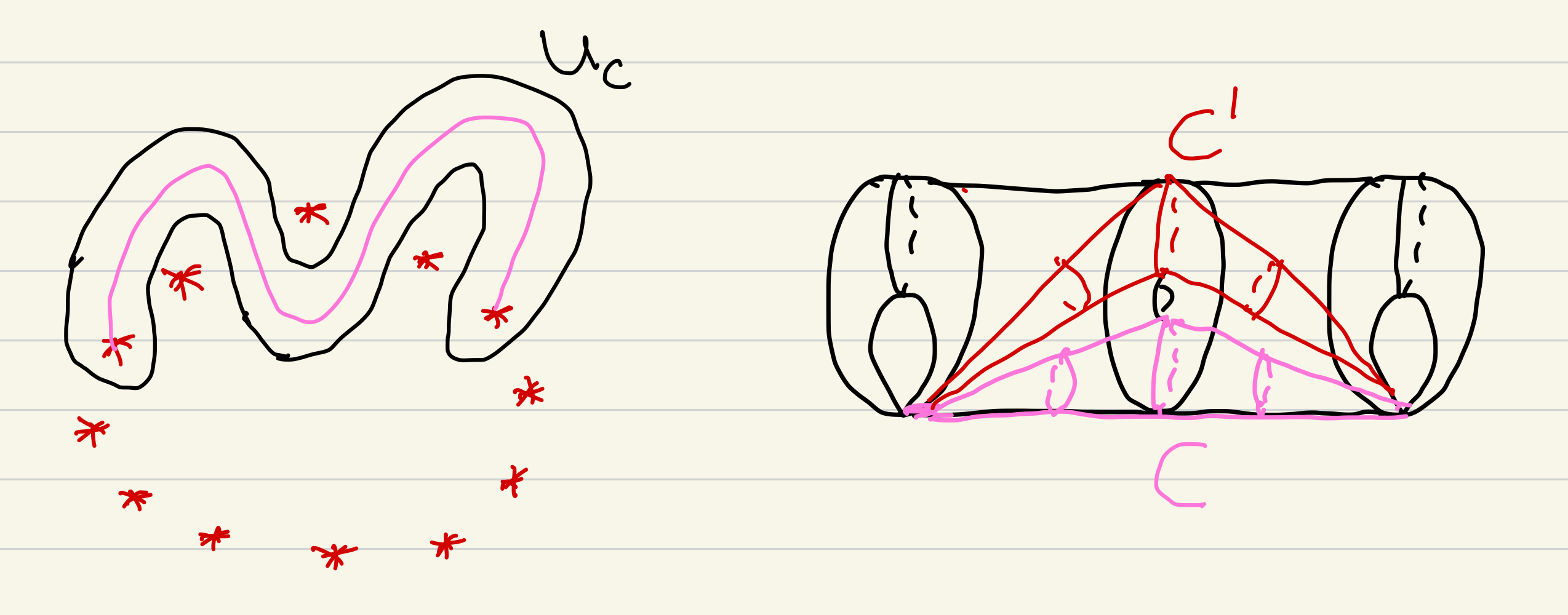}
\caption{\footnotesize On the left: each $c\in e^\perp/\ZZ e$ with $c^2=-2$ determines a pink path $\gamma_c$ between the projections to $\Pb^1$ of two nodal fibers; the disk $U_c$ is a tubular neighborhood of $\gamma_c$.  On the right: a picture of $\pi^{-1}(\gamma_c)$, together with $2$-spheres $C,C'$; these exist since the two singular fibers in $\pi^{-1}(\gamma_c)$ have a common vanishing cycle. Each of $C,C'$ is made up of two ``thimbles'', in the terminology of Lefschetz.}
\label{figure:smallsup1}
\end{figure}

\begin{figure}[h]
\includegraphics[scale=0.11]{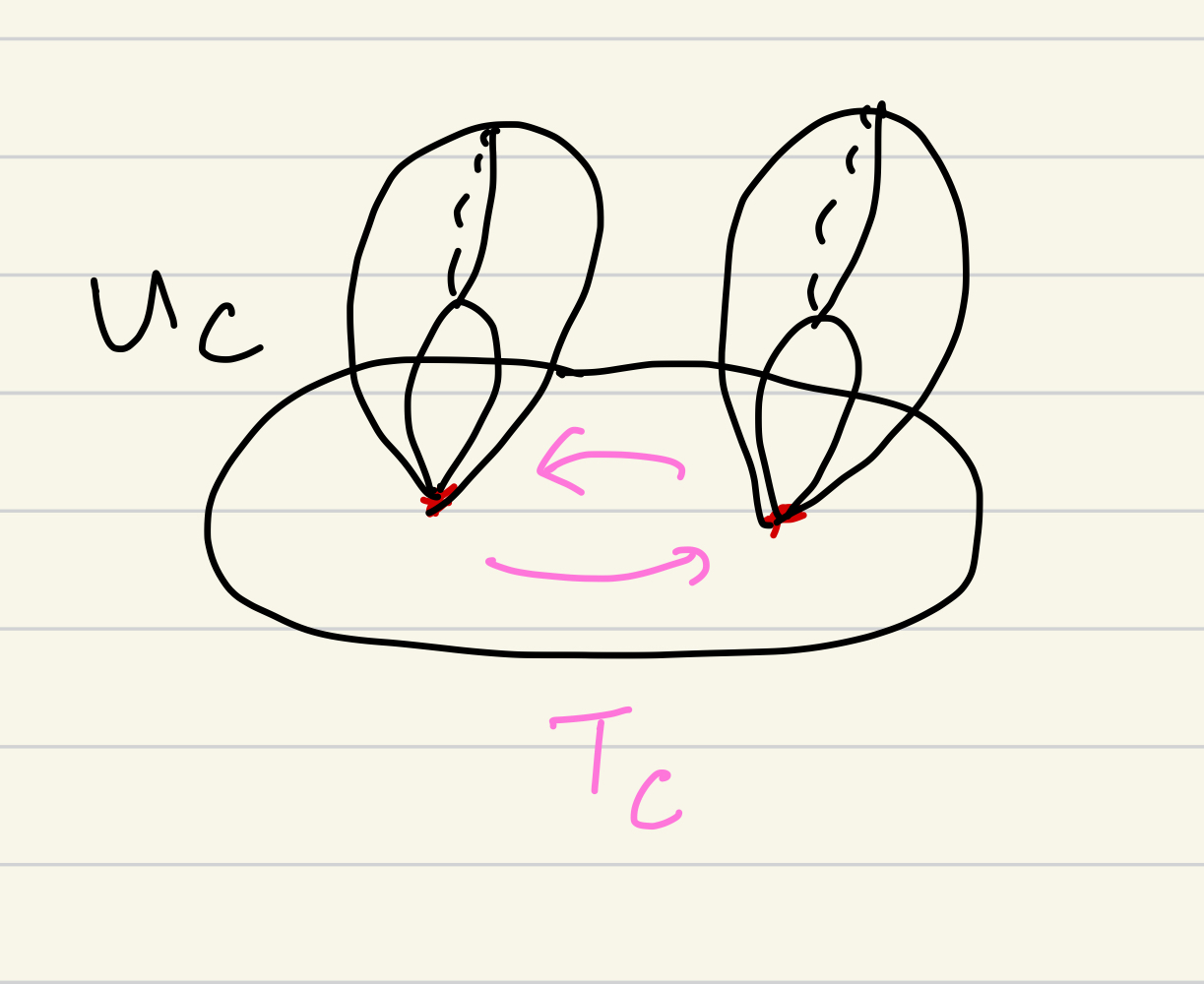}
\caption{\footnotesize In the situation described in Figure \ref{figure:smallsup1}, the Dehn twist $\tilde{\tau}(C)$ about $C$ (resp. $C'$) is a lift of an element of the spherical braid group; it ``braids'' two singular fibers of $\pi$ with the same vanishing cycle.  The product 
$\tilde{\tau}(C')\tilde{\tau}(C)^{-1}$ is isotopic to the fiber-preserving diffeomorphism $f_c$ representing the Eichler transformation $E(e\wedge c)\in \Orth(H_M)$.}
\label{figure:braid3}
\end{figure}

We prove Theorem \ref{thm:smallsupport} at the end of \S\ref{section:smallsup}.  It gives another way to realize the group 
$\Lambda_d(e)$ by a group of fiber-preserving diffeomorphisms. \

\begin{corollary}[{\bf Nielsen Realization  by Eichler transformations}]
\label{corollary:small}
For each $d\geq 1$, the lattice  $\Lambda_d(e)$ admits a basis $\Cscr$  consisting of $(-2)$-vectors.  For any such basis, the associated set 
$\{f_c\}_{c\in \Cscr}$  (in the notation of Theorem \ref{thm:smallsupport}) of fiberwise-translations determines an injective homomorphism 
$\Lambda_d(e)\to \Trans(\pi_d)$ whose composite with the projection $\Trans(\pi_d)\to \MW(\pi_d)$ is an isomorphism of lattices.
This isomorphism is the inverse of the ``Eichler representation'' of $\MW(\pi_d)$ on $\Lambda_d$.
\end{corollary}

The normal forms given in Theorem \ref{theorem:Nielsen1} and Theorem \ref{thm:smallsupport} are of two contrasting flavors: in the first, a free abelian group of fiberwise-translations of $M_d$ representing $\Lambda_d(e)$ acts freely on $M_d$; in the second, elements of $\Lambda_d(e)$ are represented by diffeomorphisms with small support (in particular that are the identity on large open sets).

\subsection{Fiber-preserving diffeomorphisms}
As an application of the results above, we are able to give a geometric realization for the full stabilizer $\Gamma_{d,e}$. To explain this, let $e\in \Lambda_d$ denote the class of a fiber of the elliptic fibration $\pi_d:M_d\to\Pb^1$, and let $\eps\in H^2(M_d)$ denote its Poincar\'e 
dual.  The first Chern class of the canonical bundle of $M_d$ (endowed with a complex structure that makes $\pi_d$ holomorphic) is $(d-2)\eps$. When $d\neq 2$ this class is nonzero and  more is true: the curves $E$ on $M_d$ with the property that  $(d-2)E$ is a canonical divisor are precisely the fibers of $\pi$. Friedman--Morgan \cite{FM2} have shown  that this class is for $d> 2$ a differentiable invariant up to sign: every self-diffeomorphism of $M_d$ preserves that class up to sign.
Subsequently, it was proved that this is the only restriction, so that $\rho_d: \Mod(M_d)\to \Gamma_d$ is surjective for  for $d\le 2$ (an older result of Borcea, \cite{Bo}); and for $d\geq 3$ is equal to the $\Gamma_d$-stabilizer $\Gamma_{d,e}$ of $e$ in $\Lambda_d$ (a more recent result of L\"{o}nne, \cite{loenne2021}).
We prove in Corollary \ref{cor:repisonto} below L\"{o}nne's result can be lifted to the mapping class group of $\pi_d$ in the following sense:

\begin{theorem}[{\bf Geometric realization}]
\label{thm:repisonto}
Each element of $\Gamma_{d,e}$ is realized by  a diffeomorphism of $M_d$ that takes elliptic fibers (of $\pi_d$) to elliptic  fibers.
\end{theorem}

This has the following remarkable consequence.

\begin{corollary}
If $d\ge 3$ then any two generic genus one fibrations  $\pi,\pi':M_d\to \PP^1$ are {\em topologically isotopic}: there exists a one-parameter family $\{h_t\in \Homeo(M_d)\}_{0\le t\le 1}$ with $h_0$ the identity and $h_1$ taking $\pi'$-fibers to $\pi$-fibers.
In particular,  every diffeomorphism of $M_d$ is topologically isotopic to a diffeomorphism that takes elliptic fibers to elliptic fibers.
\end{corollary}

\begin{proof}
Moishezon \cite{M} proved that there exists $h\in\Diff(M_d)$ 
taking $\pi'$-fibers to $\pi$-fibers.  By Theorem \ref{thm:repisonto} 
there exists $f\in\Diff(\pi)$ so that $(f\circ h)_\ast={\rm I}\in\Orth(\Lambda_d)$.  
By a recent result of Gabai-Gay-Hartman-Krushkal-Powell, building on earlier work of Kreck, Perron and Quinn (see \cite{GGHKP} and the references contained therein), $f\circ h$ 
is topologically isotopic to the identity.
\end{proof}

We further draw the reader's attention to Theorem \ref{thm:equinodal}, which refines Corollary \ref{corollary:small}. Briefly, it realizes certain Eichler transformations (that operate in $\Lambda_d$ and fix $e$) geometrically by an arc in the base $\Pb^1$ connecting two discriminant points about which to do an `anti-nodal fiber twist"

\subsection{Outline} In order to prove Theorem 
\ref{theorem:computeMW} we introduce in \S\ref{subsection:relmw}  `relative Mordell-Weil groups', which should also be a useful tool for future computations.  In \S\ref{subsection:fibsum} we find formulas for (relative) Mordell-Weil groups under fiberwise connect-sum.  A core component of this is the computation, in \S\ref{section:2node}, of the relative Mordell-Weil group of a certain compact 
$4$-manifold with nonempty boundary, equipped with a genus one fibration with two singular fibers of the `equinodal type'.  The gluing formulas then allow us to compute the differentiable Mordell-Weil group for any elliptic fibration (with all nodal fibers) of any $4$-manifold.   Note that none of the above concepts exists in the classical case of complex-holomorphic Mordell-Weil groups, as biholomorphic maps cannot fix a nonempty  open subset of $M$. Other tools used in this paper include spectral sequences,  the theory of quadratic forms and the theory of reflection groups.

\subsection{Acknowledgements} We would like to thank  R.~Hain, C.~McMullen and M.~Powell for helpful comments on an earlier version of this paper and I.~Smith  for drawing our attention to recent work of P.~Hacking and A.~Keating. We are grateful to both of these authors for subsequently explaining their paper \cite{HK} to us.  Finally, we thank  the referee, who did a meticulous job. The paper very much benefitted from the referee's comments.

\subsection{Dedication} We happily dedicate this paper to  our friend and colleague Gerard van de Geer on the occasion of his 
$75$ the birthday. This journal would not have seen the light of day  without  him.

\section{Background material}

In this section we state and prove some background results that will be used throughout the paper.

\subsection{Realizing unipotent transformations}
We begin by proving the claim made in Subsection \ref{subsect:Thurstontype} (but we will use this only for $d\le 2$).

\begin{proposition}\label{prop:isotropicfixedvector}
Any  unipotent $U\in \Orth(\Lambda_d)$ fixes some primitive isotropic vector $e\in \Lambda_d$.
\end{proposition}
\begin{proof}
Note that the fixed-point lattice  $\Lambda_d^U$ of $U$ in $\Lambda_d$ cannot be definite,  for then its orthogonal complement would contain another fixed vector.    
It then suffices to show that $\Lambda_d^U$ has rank $\ge 5$, for Meyer's theorem then implies that $\Lambda_d^U$ represents zero. 

We put $\Lambda_{d, \QQ}:=\QQ\otimes\Lambda_\QQ=H_2(M_d;\QQ)$. 
By the Jacobson-Mozorov theorem  there exists a homomorphism of $\QQ$-algebraic groups which on $\QQ$-points is given  by a $\rho: \SL_2(\QQ)\to \Orth(\Lambda_{d, \QQ})$ with 
$\rho(\begin{smallmatrix}
1 & 1\\
0 & 1
\end{smallmatrix})=U$.
We decompose $\Lambda_{d, \QQ}$ into irreducible  $\SL_2(\QQ)$-representations:
if $V_k$ stands for  the irreducible $\SL_2(\QQ)$-representation  of degree $k+1$  and $I_k$ for the space of $\SL_2(\QQ)$-homomorphisms $V_k\to \Lambda_{d, \QQ}$, then the natural map
\[
\oplus_{k\ge 0} V_k\otimes I_k\to  \Lambda_{d, \QQ}
\]
is an isomorphism of  $\SL_2(\QQ)$-representations and defines the isotypical decomposition 
$\Lambda_{d, \QQ}=\oplus_{k\ge 0} \Lambda_{d, \QQ}^{(k)}$. That decomposition  will be orthogonal with respect to the given intersection pairing  and hence induces in each $\Lambda_{d, \QQ}^{(k)}$ a nondegenerate $\SL_2(\QQ)$-invariant quadratic form. 

When  $k$ even,  $V_k$ comes with a nondegenerate $\SL_2(\QQ)$-invariant quadratic  form  of signature $(k/2, 1+k/2)$. This determines on $I_k$  a nondegenerate quadratic form
so that the isomorphism $ V_k\otimes I_k\cong \Lambda_{d, \QQ}^{(k)}$ is one of quadratic spaces.
From this  we see that $\Lambda_{d, \QQ}^{(k)}$  has a positive definite subspace of dimension $\ge  \frac{1}{2}k\dim I_k$. 

When $k$ is odd,  $V_k$ comes with a nondegenerate alternating  form.  This determines on $I_k$ a nondegenerate alternating  form such that $ V_k\otimes I_k\cong \Lambda_{d, \QQ}^{(k)}$ is an isomorphism  of quadratic spaces. In particular, the Witt index of  $\Lambda_{d, \QQ}$ is zero, and so  $\Lambda_{d, \QQ}^{(k)}$ will have  a positive definite subspace of dimension $\ge \frac{1}{2}(k+1)\dim I_k$. 

Since  $H_2(M_d;\QQ)$ is of dimension $12 d-2$ and has signature $(2d-1, 10d-1)$, we must have $\sum_k  (k+1)\dim I_k=12 d-2$ and 
$\sum_k \frac{1}{2}k \dim I_k\le 2d-1$. It follows that 
\[
\textstyle \dim I_0\ge \sum_k \big((k+1)-2k\big)\dim I_k\ge (12 d-2)-4(2d-1)=4d+2\ge 6,
\]
which implies that $\dim \Lambda_{d, \QQ}\ge 6$.
\end{proof}

\begin{corollary}\label{cor:diffrep}
Every unipotent element of $\G_2$ is realized by a diffeomorphism of $M_2$ which preserves the $\pi_2$-fibers. 
\end{corollary}
\begin{proof}
Proposition \ref{prop:isotropicfixedvector} implies that such unipotent $\varphi\in \Orth(H_2)^+$ preserves a primitive isotropic vector.
Since $\Orth(H_2)^+$  acts transitively on such vectors and is the image of $\Mod(M_2)$, we can without  loss of generality
assume that $\phi$ preserves $e$. It follows from  Theorem 1.9 of \cite{FL2} that $\varphi$ can be realized by some fiber preserving diffeomorphism.
\end{proof}

\begin{remark}\label{rem:diffrep}
The restriction in Corollary \ref{cor:diffrep} to  $d=2$ is for a reason: when $d>2$,  the fiber class $e$ is  by the theorem of Friedman-Morgan  mentioned earlier (up to sign) intrinsic to $M_d$ and hence preserved up to sign by every diffeomorphism. On the other hand  
for $d=1$ it is not true as stated, for then  $\G_1=\Orth(\Lambda_1)$ and the lattice  $\Lambda_1$ is isomorphic to both
$\Ib\perp\Ib(-1)\perp \Eb_8(-1)$ and  $\Ib\perp 9\Ib(-1)$ (odd unimodular lattices of the same indefinite signature are isomorphic to each other).  Both have  $\Ib\perp \Ib(-1)$ as their first two summands. The difference of its  two basis elements define via these isomorphisms  primitive isotropic vectors $e, e'$ of $\Lambda_1$ with  $\Lambda_1(e)\cong \Eb_8(-1)$ and $\Lambda_1(e')\cong 8\Ib(-1)$.
These represent the class of an elliptic fibration resp.\ of a conic bundle and every primitive isotropic vector of $\Lambda_1$
is $\Orth(\Lambda_1)$-equivalent to $e$ or $e'$. The orthogonal transformation which via $\Lambda_1\cong \Ib\perp 9\Ib(-1)$ has the diagonal form $(1,1,-1,\dots, -1)$ does not fix a primitive isotropic vector representing an elliptic fibration.
We make this concrete in Example \ref{example:rationalelliptic} below.
\end{remark}

\begin{example}[{\em Rational elliptic surface}]
\label{example:rationalelliptic}
Let $F$ and $F'$ be smooth cubic curves in $\PP^2$ that intersect in nine distinct points 
$p_0, \dots , p_8$. Then the   cubics in $\PP^2$ passing through $p_0, \dots , p_8$ make up a pencil that is generated by $F$ and $F'$. 

Assume that each member of the pencil is irreducible and 
is separated on the blowup $Y\to \PP^2$ of $\PP^2$ at $\{p_i\}$, so that 
$Y$ is endowed with the structure of an elliptic fibration $\pi : Y\to \PP^1$.  The fiber class  $e$ of this fibration equals  $3\ell-\sum_{i=0}^8  e_i$. Each exceptional divisor $E_i$ over $p_i$ appears as a section of $\pi$. We identify $e^\perp/\ZZ e$ with the group of Eichler transformations as in Subsection \ref{subsect:Thurstontype}. 

Since both $E_0$ and $E_i$ are (images of) sections of $\pi$, there is a unique element of 
$\MW_{\hol}(\pi)$ that takes $E_0$ to $E_i$; we therefore simply denote that element by $E_i-E_0$ ($i=1, \dots, 8$).  The element 
$E_i-E_0$ lies in $e^\perp$ and  the associated Eichler transformation gives its  action on $H_2(Y)$.
The group $H_2(Y)$ has as a basis the classes $\{e_i:=[E_i]\}_{i=0}^8$ together with the class  $\ell$ of a line in $\PP^2$ not passing through $p_0, \dots , p_8$ (viewed as a class on $Y$).
This implies that the $\{E_i-E_0\}_{i=1}^8$ are also independent as elements of $\MW_{\hol}(\pi)$.  But the latter is strictly larger (one contains the other as a sublattice of index 3): for example, 
the line through $p_1$ and $p_2$ does not contain any other $p_i$ and hence its strict transform $L_{12}$ in $Y$ is a section. Note that its class is $\ell-e_1-e_2$.
So $L_{12}-E_3$ lies in $\MW_{\hol}(\pi)$ but its class $\ell-e_1-e_2-e_3$ is not an integral linear combination of the $\{e_i-e_0\}_{i=1}^8$. It is known (see also the discussion below) that 
\[
\MW_{\hol}(\pi)\cong e^\perp/\ZZ e.
\] 
 Since taking the quotient by $\ZZ e$ allows us to eliminate $e_0$,  a basis of $e^\perp/\ZZ e$ (and hence of  $\MW_{\hol}(\pi)$) is given by the images of 
$\ell -e_1-e_2-e_3, e_1-e_2,\dots , e_7-e_8$.  This is the standard basis of the even, unimodular, rank $8$, negative-definite  lattice $\mathbf{E_8}(-1)$. 

A conic fibration as mentioned in Remark \ref{rem:diffrep} is represented by the strict transforms of the lines in $\PP^2$ passing through $p_1$. Its fiber class  is $\ell-e_1$. 
\end{example}

\subsection{The translation subgroup}
\label{subsection:trans}
Recall that an {\em affine structure}  on manifold $M$ is specified by an atlas whose coordinate changes  are affine-linear and which is maximal for that property; it is equivalent to giving a flat, torsion-free connection on the 
tangent bundle $TM$.

A closed genus one surface admits such a affine structure. When endowed with it, the resulting affine surface $E$
becomes a principal homogeneous space for  the identity component  its automorphism group. 
We denote that group  (which is  isomorphic to the  $2$-torus $T^2$) by $\Trans(E)$ 
and refer to it as the \emph{translation group} of $E$. 
A  complex structure on $E$ determines such an affine structure, because there is then a unique flat 
metric compatible with the complex structure that gives that fiber unit volume. 
Then  $\Trans(E)$ is the  identity component of the automorphism group of the Riemann surface $E$ 
and hence is identified with  the Jacobian of $E$.

The situation is similar for a once or twice punctured 2-sphere: an affine structure 
makes such a surface a torsor of  the identity component  of its automorphism 
group, which is then isomorphic  to the additive group resp.\ the multiplicative group of 
$\CC$. There are no other orientable connected surfaces with that property.

\subsection{Various mapping class groups of genus one fibrations}
In order to understand more subtle properties of fiber-preserving diffeomorphisms of $M$, we introduce some related 
subgroups of $\Diff(M)$.  Let $\pi: M\to \PP^1$ be a genus one fibration. 
Define
\begin{align*}
\Diff(\pi):=&\{F\in\Diff(M): \ F \ \text{takes fibers of $\pi$ to fibers of $\pi$}\}\\
\Diff(M/\Pb^1):=&\{F\in\Diff(\pi): F \ \text{preserves each fiber of $\pi$}\}\\
\Diff^0(M/\Pb^1):=&\{F\in\Diff(M/\Pb^1): \text{the restriction of $F$ to each smooth fiber is}\\
&\text{isotopically trivial}\}\\
\end{align*}
For a singular (Kodaira)  fiber of $\pi$, the affine structure extends to the part  where the fiber is smooth (so \emph{a fortiori} of multiplicity  one). The fiberwise translations then extend to this fiber  by a certain group multiplication. So
\begin{equation}
\label{eq:sequence1}
\Trans(\pi)\subset \Diff^0(M/\Pb^1)\subset \Diff(M/\Pb^1)\subset\Diff(\pi)\subset\Diff(M)
\end{equation}

Applying $\pi_0$ to each group in \eqref{eq:sequence1} gives a string of homomorphisms
\begin{equation}
\label{eq:sequence2}
\MW(\pi)\to\Mod^0(M/\Pb^1)\to\Mod(M/\Pb^1)\to
\Mod(\pi)\to\Mod(M).
\end{equation}
Note that the image of $\Mod^0(M/\Pb^1)$ in $\Orth(H_M)$ lies in $e^\perp/\ZZ e$ for $e\in H_M$ the fiber class of the fibration $\pi$.

\begin{remark}\label{rem:index2}
If $\pi:M\to \PP^1$ is an elliptic fibration, then  $\Mod^0(M/\Pb^1)$  has index two in $\Mod(M/\Pb^1)$. In that case the nontrival coset is represented by the  involution that acts in each fiber as `minus the identity' relative to the given section. If all fibers are integral,  then such an involution acts in $e^\perp/\ZZ e$ as minus the identity.
\end{remark}

Let $\pi :X\to B$ be a (locally trivial) fiber bundle in the smooth category whose 
fibers are closed genus one surfaces.  Assume that both $X$ and $B$ are oriented; this will 
then  orient every fiber.  In case each fiber comes  endowed with an affine structure 
(depending smoothly on the base point), then  the structure group of $\pi$ is the 
semi-direct product  $T^2\rtimes \SL_2(\ZZ)$. This determines a subgroup 
$\Diff^+_\aff (\pi)\subset \Diff^+(\pi)$ as the subgroup of preserving this  affine structure.
The subgroup that in addition preserves each fiber, 
\[
\Diff^+_\aff(X/B):=\Diff^+(X/B) \cap \Diff^+_\aff (\pi),
\] 
contains 
the group $\Trans(X/B)$ of  fiberwise translations as a normal, abelian  subgroup.  We have denoted the  connected component group of $\Trans(X/B)$ by $\MW(\pi)$; for any of the other diffeomorphism groups,  we denote its connected component group by replacing 
$\Diff$ by $\Mod$. 

We will need the following ``families version'' of a theorem of Earle-Eells.  Let $\Diff^0(X/B)$ 
 denote the subgroup of $\Diff(X/B)$ of diffeomorphisms that induce in each fiber $X_b$ the identity on $H_1(X_b)$.  (Warning: 
 $\Diff^0(X/B)$ need not be the identity component of $\Diff(X/B)$).
 
\begin{proposition}
\label{prop:trans2}
 The inclusion $\Trans(\pi)\subset \Diff^0(X/B)$ induces a surjection on arcwise connected components. 
\end{proposition}

\begin{proof}
This is known to be so when $B$ is a point; in fact, a theorem due to 
Earle-Eells (\cite{EE}, last Corollary of \S 11) asserts that there exists a strong deformation retraction 
\[r: [0,1]\times \Diff^0(T^2)\to \Diff^0(T^2)\] of $\Diff^0(T^2)$  onto the translation group $\Trans(T^2)$ (so that it is a copy of $T^2$),  i.e., 
$r_t$ is the identity on  $\Trans(T^2)$ for all $t\in [0,1]$,  $r_0$ is the identity and $r_1$ has image $\Trans(T^2)$. 

Choose a smooth triangulation $\Sigma$ of $B$ whose vertex set contains the discriminant.
Let $f \in\Diff^0(X/B)$. The first step is to find a path in $\Diff^0(X/B)$ that connects $f$ with an $f_0$ that is
the identity over a regular neighborhood of the $0$-skeleton of $\Sigma$. For the smooth fibers we can 
invoke the theorem of Earle-Eells, but the singular fibers need special care:  
first choose a path in $\Diff^0(X/B)$ that connects $f$ with an $f'_0$ that is the identity near the singular points of the fibers. Then $f'_0$ induces in the smooth part of each singular fiber (which is a copy of $\CC$ or $\CC^\times$) a compactly supported  diffeomorphism. It  is well-known that the compactly supported  diffeomorphisms of such a surface have the identity element as a strong deformation retract. The construction of $f_0$ is then straightforward.

The rest of the argument follows a standard  pattern: if  $f_{k-1}\in \Diff^0(X/B)$ is  a fiberwise translation on a regular neighborhood of the $(k-1)$-skeleton for some  $k\in \{1,2\}$, then we prove that  there exists a path in $ \Diff^0(X/B)$ from $f_{k-1}$ to some $f_k$ that  is a fiberwise translation over a regular neighborhood of the $k$-skeleton. This will of course imply the proposition. 
If $\sigma$ and  a $k$-simplex of $\Sigma$ and we choose a trivialization $X_{|\sigma|}\to T^2$  of $\pi_{|\sigma|}: X_{|\sigma|}\to |\sigma|$ as an affine family, then $f_{k-1}$ defines a map
$ |\sigma|\to \Diff^0(T^2)$ which takes values in $\Trans(T^2)$ on a neighborhood of the boundary. The Earle-Eells theorem then provides 
the  desired path over $\sigma$ from
$f_{k-1}$ to a map taking its values in $\Trans(X/B)$. We omit the details.
\end{proof}

\subsection{Translations near a Kodaira fiber}
We now recall what happens near a singular fiber. Assume that  $B$ is an open complex disk, $o\in B$ and that  the genus one fibration 
$\pi:X\to B$ has  $X_o$ 
as its unique singular fiber (of Kodaira type). The inclusion $X_0\subset X$ is then a deformation retract and in particular a homotopy equivalence.  
Denote by  $\Trans(X_o)$ the group of  automorphisms of 
$X_o$ that extend to a $B$-automorphism of $X$ acting  as  fiberwise translations in the smooth fibers. 

The Kodaira classification shows that the group $\Trans(X_o)$
acts transitively on the part of $X_o$ where it is smooth (we use here smooth in the sense 
of algebraic geometry and so this implies reduced). 
This smooth part need not be connected and hence the group 
$\pi_0(\Trans(X_o))$ of connected components, which is isomorphic to 
$\Trans(X_o)$ modulo its identity component, may be nontrivial. In fact, this torsion group is faithfully represented by its action 
on the intersection graph of the Kodaira fiber 
(which is an affine Dynkin): it permutes  simply transitively the vertices that  have coefficient 1 in the fiber class. In terms of the associated finite root system,  $\pi_0(\Trans(X_o))$ is  identified with the weight lattice modulo the root lattice. The latter  has here the more concrete description  as a discriminant group:  the class  $e\in H_2(X)$ generates the radical of the intersection pairing on $H_2(X)$, so that $\overline H_2(X):=H_2(X)/\ZZ e$ is  nondegenerate  and 
then $\pi_0(\Trans(X_o))$ is naturally identified with the finite abelian group
\[
\discr H_2(X):=\overline H_2(X)^\vee/\overline H_2(X).
\]
This group is trivial if and only when $X_o$ is integral (a nodal or a cuspidal curve, but of course also when it is smooth) or is of Kodaira type $\text{II}^*$
(as $\hat E_8$-fiber).
The identity component of $\Trans(X_o)$ is isomorphic to the multiplicative group of $\CC$ in case $X_o$ is of type $\textrm{I}_r$  (the affine Coxeter group is then of affine type $\hat A_{r-1}$) and to the  additive  group of $\CC$ otherwise.

For any genus one fibration $\pi :X\to B$  we thus get,  by assigning to an open subset $U\subset B$ the group of fiberwise translations  over $U$, a  sheaf of abelian topological groups  $\Tscr_\hol(\pi)$ on $B$. We define its 
$C^\infty$-counterpart as  $\Tscr(\pi):=\Ecal_B\otimes_{\Ocal_B}\Tscr_\hol(\pi)$, where $\Ecal_B$ is the sheaf of smooth 
$\RR$-valued functions on $B$. So here the translation may depend in 
a differentiable manner on the base point. 
We will write $\Trans(\pi)$ for the group of global sections of $\Tscr(\pi)$.

\section{The differentiable Mordell-Weil group }
\label{section:MW}
In this section we introduce a smooth version of the holomorphic Mordell-Weil group.  This group already made an implicit  appearance in the preceding section, but here we give the formal definition and develop this notion in a more systematic manner.

\subsection{Smooth Mordell-Weil groups} In the rest of this section $B$ will always stand for  an \emph{oriented} smooth surface of finite type with compact (possibly empty) boundary $\p B$ and interior $\mathring B$. If it has  a complex structure, we will explicitly say so.

The preceding  leads to the following notions. A \emph{genus one fibration} over $B$  is
a proper differentiable map $\pi: X\to B$, where $X$ is an oriented  $4$-manifold with boundary $\p X=\pi^{-1}\p B$ of which each  smooth fiber comes with an affine structure structure smoothly depending on the base point  that makes that fiber a torsor for  a torus. We demand that the discriminant $D_\pi$ of $\pi$ (i.e., the set of critical values) is contained in $\mathring B$ and that each of its points has a neighborhood  over which  $\pi$ is isomorphic to a Kodaira degeneration by an orientation preserving  diffeomorphism that is fiberwise affine. In that case, the sheaf $\Tscr(\pi)$ and its group $\Trans(\pi)$ of global sections is still defined.
\begin{definition}[{\bf Relative Mordell-Weil group}]
\label{def:MWdiff} 
With notation as above, the (smooth) \emph{Mordell-Weil group}  $\MW(\pi)$ is the group of connected components of the group $\Trans(\pi)$ of diffeomorphisms of $B$ that are fiberwise translations (as this is clearly an abelian group, we write this group additively). 

More generally, if $K\subset B$ is a closed subset, then
$\MW(\pi, \pi_K)$ stands for  the group of connected components of the group of diffeomorphisms of $M$ that are fiberwise translations and are the identity over a neighborhood of $K$; in case $K=\p B$, we often write 
$\p\pi$ for $\pi_K$, the map $\pi^{-1}K\to K$ induced by $\pi$.
\end{definition}

If there are no singular fibers (so that we have an affine  torus bundle), then the above definition makes of sense for any base manifold. With this mind, it is not hard to see that there is 
an exact sequence
\[
\MW(\pi, \p\pi)\to \MW(\pi)\to \MW(\p\pi).
\]
Although our main interest will be when the singular fibers are all nodal, we also sometimes  have to deal with other Kodaira fibers. 
We note that a section $\sigma$ of $\pi$ will meet every fiber transversally and will do so in a smooth point of that fiber (so that the component on which lies must have multiplicity one in the fiber class): a local equation for the fiber over $p\in S$,
pulled back along $\sigma$, will have a simple zero at $s$.

\subsection{A cohomological characterization of $\MW(\pi)$} 
Let $\pi: X\to B$ be as above. 
If $X_b=\pi^{-1}b$ is smooth then the translation group  $\Trans (X_b)$ is naturally isomorphic to 
the torus $H_1(X_b; \RR)/H_1(X_b; \ZZ)$, which we identify via Poincar\'e duality with 
$H^1(X_b; \RR)/H^1(X_b; \ZZ)$. This almost remains true if $X_b$ is a nodal curve with singular point $p$, for then 
\begin{multline*}
\Trans (X_b)\cong H_1(X_b\ssm \{p\}; \CC)/H_1(X_b\ssm \{p\}; \ZZ)\cong\\\cong  
H^1(X_b,\{p\}; \CC)/H^1(X_b, \{p\}; \ZZ) \cong H^1(X_b; \CC)/H^1(X_b; \ZZ)\cong\CC^\times
\end{multline*}
where the second isomorphism is given by Alexander duality.  This group contains the circle group $H^1(X_b; \RR)/H^1(X_b)$ as a maximal compact subgroup, and the inclusion is a homotopy equivalence. 

In case $X_b$ is a cuspidal curve, $\Trans (X_b)\cong \CC$ and $H^1(X_b; \CC)/H^1(X_b)$ is trivial, so if we think of $H^1(X_b; \CC)/H^1(X_b)$ as the identity element of $\Trans(X_b)$, then this still a homotopy equivalence.
This shows that if  $ \Tscr^c(\pi)$ denotes
the subsheaf  of  $\Tscr(\pi)$ for which the translations belong to a maximal compact subgroup, then  the quotient $\Tscr(\pi)/\Tscr^c(\pi)$ has its support on the discriminant $D_\pi$ with fibers copies of $\RR$ or $\CC$. This implies that $\Tscr^c(\pi)\subset \Tscr^c(\pi)$ induces a map $H^q(B,\Tscr^c(\pi))\to H^q(B,\Tscr(\pi))$ that is an  isomorphism for $q>0$  and  
induces for $q=0$ an isomorphism on  their  connected component groups (note that these spaces of sections are topological groups).

\begin{remark}\label{rem:sheafconvention}
In what follows we adhere to topological conventions when dealing with sheaf cohomology:  if  $(B, C)$ is a topological pair   with $C$ closed in $B$ and $\Fscr$ is an abelian  sheaf on the open subset $B\ssm C$, then $H^q(B,C; \Fscr)$  stands for $H^q(B; j_!\Fscr)$, where $j:B\ssm C\subset B$ and $ j_!\Fscr$ is the extension of $\Fscr$ to $B$ by zero. This means that if $i: C\subset B$ and 
$\Gscr$ is an abelian sheaf on $B$, then there is a long  exact sequence 
\[
\cdots\to  H^{q}(B,C; j^*\Gscr)\to H^q(B, \Gscr)\to H^q(C, i^*\Gscr)\to H^{q+1}(B,C; j^*\Gscr)\to \cdots 
\]
We here often abuse notation a bit by suppressing $j^*$ and $i^*$ in the coefficient sheaves.
\end{remark}

\begin{proposition}[{\bf Cohomological characterization of MW groups}]
\label{prop:MWcohomology}
 Assume that $\pi$  has only integral fibers  (all singular  fibers are either nodal or cuspidal). Then there are natural isomorphisms 
 \[
 \MW(\pi)\cong H^1(B, R^1\pi_*\ZZ) \ \ \text{and}\ \  \MW(\pi, \p\pi) \cong H^1(B, \p B; R^1\pi_*\ZZ).
 \]
 Moreover,  for $q>0$, there are natural isomorphisms 
\[
H^q(B, \Tscr(\pi))\cong  H^{q+1}(B,R^1\pi_*\ZZ), \quad 
H^q(B,\p B;\Tscr(\pi))\cong  H^{q+1}(B,D_\pi;R^1\pi_*\ZZ).
\]
\end{proposition}

\begin{proof}
There is (essentially by definition) a  short exact sequence 
\begin{equation}\label{eqn:transl}
0\to R^1\pi_{*}\ZZ\to \Escr_B\otimes_\ZZ R^1\pi_{*}\ZZ \to \Tscr^c(\pi)\to  0.
\end{equation}
Since $\Escr_B$ is a soft sheaf, so is $\Escr_B\otimes_\ZZ R^1\pi_*\ZZ $ and hence $H^q(B, \Escr_B\otimes_\ZZ R^1\pi_*\ZZ)=0$ for $q>0$. It follows that the long exact cohomology sequence associated to \eqref{eqn:transl} gives the exact sequence of topological groups
\begin{multline*}
0\to H^0(B,R^1\pi_*\ZZ)\to H^0(B, \Escr_B\otimes_\ZZ R^1\pi_*\ZZ)  \to  H^0(B,\Tscr^c(\pi))\to H^1(B,R^1\pi_*\ZZ)\to 0
\end{multline*}
and that for $q>0$:  
\[ H^q(B, \Tscr^c(\pi))\cong H^q(B, \Tscr(\pi))\cong  H^{q+1}(B,R^1\pi_*\ZZ).\]
Since $H^0(\pi, \Escr_B\otimes_\ZZ R^1\pi_*\ZZ)$ is a vector space (hence connected)
and $H^1(B,R^1\pi_*\ZZ)$ is discrete, it follows that 
$\MW(\pi)=\pi_0H^0(B, \Tscr(\pi))\cong \pi_0H^0(B,\Tscr^c(\pi))$ gets identified with $H^1(B,R^1\pi_*\ZZ)$.  The assertions regarding  $\MW(\pi, \p\pi)$ and
$H^q(B, \p B;\Trans(\pi))$ are obtained similarly.
\end{proof}

Proposition \ref{prop:MWcohomology} shows that in that situation (all fibers integral)  the groups  $\MW(\pi)$ and $\MW(\pi, \p\pi)$ are finitely generated if the discriminant $D_\pi$ is  finite.  According to the following theorem, this is true in general.

\begin{theorem}\label{thm:finitegen}
Assume the discriminant $D:=D_\pi$ finite. Then there are short exact sequences 
\begin{gather*}
H^1(B, D; R^1\pi_*\ZZ)\to  \MW(\pi)\to \oplus_{b\in D} \discr H_2(X_b)\to 0,\\
H^1(B, \p B\cup D;R^1\pi_*\ZZ)\to  \MW(\pi, \p\pi)\to \oplus_{b\in D} \discr H_2(X_b)\to 0.
\end{gather*}
In particular, $\MW(\pi)$ and $\MW(\pi, \p\pi)$ are finitely generated.
\end{theorem}

Theorem \ref{thm:finitegen} in particular implies Theorem \ref{theorem:smoothMW4}.

\begin{proof}[Proof of Theorem \ref{thm:finitegen}]
We only do this for $\MW(\pi)$; the proof for $\MW(\pi, \p\pi)$ is similar. Let $D\subset B$ stand for the discriminant of $\pi$. Then there is a natural map 
$\MW(\pi, \pi_D)\to \MW(\pi)$. For every fiber $X_b$
(in particular, for a singular one),  the connected component group of the stalk $\Tscr(\pi)_b$ is a torsion group which acts simply transitively on the set of irreducible components of $X_b$ with multiplicity one. The evident map $ \MW(\pi)\to \oplus_{b\in D} \pi_0 \Tscr(\pi)_b$ is surjective and has as kernel the connected component group of the group of sections that are vanish  on $D$.

The sheaf $R^1\pi_*\ZZ$ is constructible. This implies that $H^1(B, D; R^1\pi_*\ZZ)$ (which is the same as 
$H^1_c(B\ssm D; R^1\pi_*\ZZ)$) is finitely generated. 
Since  $ \oplus_{b\in D} \discr H_2(X_b)$ is also  finitely generated, 
$\MW(\pi)$ will be as well.
\end{proof}

\subsection{Sections of $\pi$ and the action of $\MW(\pi)$ on $H^2(X)$}
If $\pi$ admits a section $\sigma$ then the space $\G(\pi)$ of sections of $\pi$ is a torsor under the group of fiberwise translations.  In this case 
the set  $\pi_0(\G(\pi))$ of connected components of $\G(\pi)$  is an $\MW(\pi)$-torsor. But sections of $\pi$ need not exist. 
 If $\sigma\in H^2(X)$ is the class of a section and $e\in H_2(X)$ is the image of the fundamental class of a general fiber,   then $\la \sigma | e\ra =1$ and so a necessary condition is that the linear form $\alpha \in H^2(X)\mapsto \la\alpha |e\ra \in \ZZ$ takes the value $1$. This is equivalent to $e$ being \emph{primitive} in the sense that it spans a copy of $\ZZ$ as a direct summand of $H_2(X)$. 

The following theorems gives (among other things) a sufficient condition for the existence of a section.  

\begin{proposition}[{\bf Existence of sections and action of $\MW(\pi)$ I}]
\label{thm:MWcohomologyI}
Let  $\pi :X\to B$ be a genus one fibration over  a compact, connected, oriented surface $B$ with nonempty boundary 
$\p B$.  Assume that each fiber of $\pi$ is integral.  Denote by $e\in H_2(X)$ the homology class of a general fiber and  by  $H^2(X)^e$ its annihilator 
in $H^2(X)$ (i.e., the cohomology classes that vanish on $e$). 

   Then $\pi$ admits a section. For every  $\tau\in \MW(\pi)$ there exists a unique $\g_\tau\in H^2(X)$ such that for all $\xi\in H^2(X)$: 
\[\tau(\xi)= \xi+\la \xi|e\ra \g_\tau\]
and the resulting map $\tau\in \MW(\pi)\mapsto  \g_\tau\in H^2(X)^e$ is an isomorphism of abelian groups.
\end{proposition}

\begin{proposition}[{\bf Existence of sections and action of $\MW(\pi)$ II}]
\label{prop:MWcohomologyII}
Let  $\pi :M\to \PP^1$ be a genus one fibration of arithmetic  genus $d\ge 1$ with integral fibers (so $M\cong M_d$).
Then $\pi$ admits a section and for every  $\tau\in \MW(\pi)$ there exists a unique $c_\tau\in e^\perp/\ZZ e$
(where  $e\in H_2(M)$ is the fiber class)
 such that $\tau$ acts on $H_2(M)$ as the Eichler transformation $E(e \wedge c_\tau)$ given by 
\[
E(e\wedge c_\tau)(x)=x+ (x \cdot e )\hat c_\tau - (x\cdot \hat c_\tau)e -\tfrac{1}{2}(c_\tau\cdot c_\tau)(x\cdot e)e \]
for all $x\in H_2(M)$, 
where $\hat c_\tau\in e ^\perp$ is a lift of $c_\tau$. 
The resulting map 
\[
\tau\in\MW(\pi)\to  c_\tau\in e^\perp/\ZZ e
\]
 is an isomorphism of abelian groups. In particular, $\MW(\pi)$ has  naturally the structure of a lattice (which is even unimodular and of 
 signature $(2d-2, 10d-2)$). 

\end{proposition}

\begin{remark}\label{rem:}
Since  $e^\perp/\ZZ e\cong \Lambda_d(e)$ is an even lattice ,   $\tfrac{1}{2}(c_\tau\cdot c_\tau)\in \ZZ$ and so $E(e\wedge c_\tau)$ preserves $H_2(M)$. 
\end{remark}

Before we get into the proofs, we note that Proposition \ref{prop:MWcohomologyII} implies Theorem \ref{theorem:computeMW}:

\begin{proof}[Proof of Theorem \ref{theorem:computeMW}]
Choose a basis of the lattice $\MW(\pi_d)$ and represent  each basis element by a translation. Since $\Trans(\pi_d)$ is abelian, this 
determines a group homomorphism $\MW(\pi_d)\to\Trans(\pi_d)$ that is a section of  $\Trans(\pi_d)\to \MW(\pi_d)$.
\end{proof}

Since the proofs of these propositions have arguments in common, we prove them almost simultaneously. Therefore, 
in the setting of Proposition \ref{prop:MWcohomologyII}, we write $B$  for $\PP^1$ and $X$ for $M$.

\begin{proof}[Proof of  Propositions  \ref{thm:MWcohomologyI} and \ref{prop:MWcohomologyII}]
We prove the propositions  in a number of steps.
Our assumption that the fibers of $\pi$ are integral implies that 
$R^2\pi_*\ZZ$ is  a trivial local system of rank one on $B$ so that a natural isomorphism  $H^0(B, R^2\pi_*\ZZ)\cong \ZZ 
$ is given by integration over (=natural pairing with) the fiber class, i.e., the map $\xi\in H^2(X)\mapsto \la \xi|e\ra \in \ZZ$.
\\

\emph{Step 1: The Mordell-Weil group $\MW(\pi)$ acts trivially on the grading of $H^2(X)$ with respect to  the Leray filtration defined by $\pi$. }

Consider the Leray spectral sequence
\[
E_2^{p,q}= H^p(B, R^q\pi_*\ZZ)\Rightarrow H^{p+q}(X).
\]  
Any diffeomorphism of $X$ that is a fiberwise-translation acts trivially on each term $E_2^{p,q}$. Hence the same is true for  its action  on $E_\infty^{p,q}$, which is by definition the grading of $H^2(X)$ with respect to  the Leray filtration.  
\\

\emph{Step 2: If $H^{2}(B,R^1\pi_*\ZZ)$ vanishes then $\pi$ admits a section.}

Since $\pi$ has no multiple fibers, we can find  an open cover $\Uscr$ of  $B$ such that  for each $U\in \Uscr$, $\pi_{U}$ admits a section $s_U$. Then $s_{U'}-s_U$ defines a section of $\Trans(\pi)$ over $U\cap U'$. Together they make up a \v{C}ech $1$-cycle relative the covering $\Uscr$ and hence define an element of $\check{H}^1(\Uscr,\Trans(\pi))$. If this a coboundary, then
there exist $\{t_U\in H^0(U, \Trans(\pi))\}_{U\in \Uscr}$ such that $s_{U'}-s_U=t_{U'}-t_U$, which means that the collection $\{s_U+t_U\}_{U\in\Uscr}$ defines a global section. Since we are allowed to pass to a refinement of $\Uscr$, it suffices to show that $\check{H}^1(B,\Trans(\pi))$ is trivial. For paracompact spaces \v{C}ech cohomology is ordinary cohomology, and so this  is  by Proposition \ref{prop:MWcohomology} equal to $H^{2}(B,R^1\pi_*\ZZ)$. 
Hence Step 2 follows.
\\

\emph{Step 3: Proof  of  Proposition \ref{thm:MWcohomologyI}.}

Then $B$ contains an embedded  graph $G\subset B$ as a deformation retract which  contains all the  critical values. This implies that the restriction map  $E_2^{p,q}= H^p(B, R^q\pi_*\ZZ)\to H^p(G, R^q\pi_*\ZZ)$ is an isomorphism  so that $E_2^{p,q}$ is trivial unless 
$p\in \{0,1\}$ and $q\in \{0,1,2\}$.  In particular, $\pi$  admits a section by Step 2. It also follows that  the Leray sequence degenerates on this page, so that  there is an exact sequence
\[
0\to H^1(B, R^1\pi_*\ZZ)\to H^2(X)\to H^0(B, R^2\pi_*\ZZ)\to 0.
\]
The  map $H^2(X)\to H^0(B, R^2\pi_*\ZZ)\cong \ZZ$ is given by integration over the fiber class $e$. It is 
surjective, because it takes the value $1$ on a section (we here regard the class of a section as an element of 
$H_2(X, \p X)\cong H^2(X)$). Proposition \ref{prop:MWcohomology}  identifies  $\MW(\pi)$ with its kernel of 
this map.  Any  $\tau\in \MW(\pi)$ induces a transformation of $H^2(X)$ which preserves the above short 
exact sequence  and acts trivially on both $H^1(B, R^1\pi_*\ZZ)$ and $H^0(B, R^2\pi_*\ZZ)\cong\ZZ$. 
Hence it is of the form $\xi\in H^2(X)\mapsto \xi +\la \xi\, |\, e\ra \g_\tau$ for some 
$\g_\tau\in H^1(B, R^1\pi_*\ZZ)$. So if  $\xi$ is the class  of a section, then its image under $\tau$ is $\xi+\g_\tau$. 
The definition shows that $\tau\mapsto \g_\tau$ is in fact the isomorphism 
$\MW(\pi)\cong H^1(B, R^1\pi_*\ZZ)$ found above.
\\

\emph{Step 4: Proof  of  Proposition \ref{prop:MWcohomologyII}}

We first show that the differential $\ZZ\cong H^0(\PP^1, R^2\pi_*\ZZ)\to H^2(\PP^1, R^1\pi_*\ZZ)$ in the Leray 
spectral sequence is the zero map. Indeed, we already observed that integration over $e$ defines 
an isomorphism $H^0(\PP^1, R^2\pi_*\ZZ)\cong \ZZ$.
Since the Leray spectral sequence degenerates on its third page, the kernel of this differential is  a 
quotient of $H^2(M)$ and it is via this quotient that the integration map 
$\xi\in H^2(M)\mapsto \la\xi | e\ra$ is defined. Since $e$ is primitive, the latter is onto and hence this kernel 
is all of  $H^0(\PP^1, R^2\pi_*\ZZ)$. So the differential $H^0(\PP^1, R^2\pi_*\ZZ)\to H^2(\PP^1, R^1\pi_*\ZZ)$ must be \
the zero map. 

This implies that $H^2(\PP^1, R^1\pi_*\ZZ)$ embeds in $H^3(M)$. But $H^3(M)\cong H_1(M)=0$ and hence $H^2(\PP^1, R^1\pi_*\ZZ)=0$.
So $\pi$ admits a section by Step 2. This in turn implies that the natural map $H^2(\PP^1)\to H^2(M)$ is injective.
Hence the Leray filtration  of  $H^2(M)$ is given by 
\[
H^2(M)=L_2H^2\supset L_1H^2\supset L_0H^2=H^2(\PP^1)\supset L_{-1}H^2=0
\]
with $L_1H^2/L_0H^2\cong H^1(\PP^1, R^1\pi_*\ZZ)$ and $L_2H^2/L_1H^2\cong  H^0(\PP^1, R^2\pi_*\ZZ)$. 
The preceding shows that $L_1H^2$ is the annihilator of $e$ and that $L_0H^2=H^2(\PP^1)$ is spanned by the Poincar\'e dual $\eps $ of $e$. If $\sigma\in H_2(M)$ is the class of  a section, then its Poincar\'e dual and $\eps $ span a unimodular lattice of rank $2$  of signature  $(1,1)$ (even or odd according the parity of $\sigma\cdot\sigma=-d$).  
Since the intersection pairing on $H_2(M)$ is unimodular, its orthogonal complement is also unimodular. This 
orthogonal complement maps isomorphically onto $\eps ^\perp/\ZZ \eps $ and hence the latter  is unimodular. 

Any orthogonal transformation of the unimodular lattice  $H_2(M)$ whose induced transformation on $H^2(M)$ acts trivially on the successive quotients of the above filtration (defined by the isotropic vector $\eps$) is an Eichler transformation as stated in the theorem  with $c$ canonically defined in $\Lambda_d(e)$. In particular, we find for every $\tau\in \MW(\pi)$ an element $c_\tau\in \Lambda_d(e)$. For $\sigma$  as above,  $\tau (\sigma)\equiv \sigma+c_\tau\pmod{\ZZ \eps }$.
Again the  definition shows that $\tau\mapsto c_\tau$ is in fact the isomorphism $\MW(\pi)\cong H^1(\PP^1, R^1\pi_*\ZZ)\cong e^\perp/\ZZ e$. 
\end{proof}

\begin{remark}\label{rem:}
In the K\"ahler setting, the lattice $\eps ^\perp/\ZZ \eps $ has a natural Hodge structure of weight $2$, and a theorem of Shioda \cite{shioda} then implies that the torsion-free quotient of $\MW_\hol(\pi)$ 
can be identified with the $(1,1)$-part of this lattice. In fact, he shows  that in general (where we admit reducible fibers), it is the quotient of $\eps ^\perp$ by  the span of the cohomology classes supported by the singular fibers.
\end{remark}

\begin{corollary}[{\bf Smooth Mordell-Weil group of a rational elliptic fibration}]
\label{cor:}
If $\pi :M\to \PP^1$ is a rational elliptic fibration  with all fibers integral, then 
its  holomorphic Mordell-Weil group $\MW_\hol(M)$ maps isomorphically to its smooth counterpart $\MW(\pi)$. In particular, the latter is a lattice is isometric to $\mathbf{E_8}(-1)$.
\end{corollary}

As observed in  Remark \ref{remark:comparison1}, this is no longer true  when $d\ge 2$. 

\begin{proof}
As Example \ref{example:rationalelliptic} shows, the holomorphic Mordell-Weil group is identified with $e ^\perp/\ZZ e$. The corollary then follows from Proposition \ref{prop:MWcohomologyII}.
\end{proof}

With Proposition \ref{prop:MWcohomologyII} in hand, we can now prove the following.

\begin{proof}[Proof of Theorem \ref{theorem:Nielsen1}]  
By Proposition \ref{prop:MWcohomologyII} there is an isomorphism $\MW(\pi_d)\cong \Lambda_d(e)$. Now choose a basis of $\Lambda_d(e)$  
and lift  that basis  to a group of fiberwise translations in $M_d$. Then the group generated by these is abelian and can be considered as  Nielsen realization of $\MW(\pi_d)$ (and of its image in $\Orth(\Lambda_d)$).
\end{proof}

\subsection{Relative Mordell-Weil groups}
\label{subsection:relmw}
We continue with the setting and notation of Proposition \ref{thm:MWcohomologyI}: $\pi:X\to B$ is a genus one fibration over a 
compact, connected surface with nonempty boundary with integral fibers and  no singular fibers over $\p B$. 
We denote $e$ its fiber class (that we consider as an element of $H_2(\p X)$ but identify with its image in $H_2(X)$) and by $H^2(X)^e$ its annihilator in $H^2(X)$ as before.
It is clear that the natural map $H_2(X)\to H_2(X, \p X)\cong H^2(X)$ has its image in $H^2(X)^e$ and factors through $H_2(X)/\ZZ e$.

\begin{proposition}\label{prop:MWpair}
Assume  we are in the setting of Proposition \ref{thm:MWcohomologyI}. Then the sequence 
\[
 H_2(X)/\ZZ e\to H^2(X)^e\to H^2(\p X)^e,
\]
where the first map  is a factor of the map $H_2(X)\to H^2(X)$ defined by the intersection pairing,  is exact and is naturally isomorphic to 
\begin{equation}\label{eqn:MWsequence}
\MW(\pi,\p\pi)\to \MW(\pi)\to \MW(\p\pi).
\end{equation}
In particular, $\MW(\pi,\p\pi)$ is free abelian. 

If in addition $H_1(X)=0$ and $\p B $ is connected (so diffeomorphic to a circle), then $\MW(\pi)\to \MW(\p\pi)$ is onto and  the above sequence is part of the (self-dual) exact sequence 
\[
0\to H^1(X_b)^{\p B}\to H_2(X)/\ZZ e\to H^2(X)^e\to H_1(X_b)_{\p B}\to 0, 
\]
where $b\in \p B$ and $H^1(X_b)^{\p B}$ resp.\  $H_1(X_b)_{\p B}$ stands for the  invariant cohomology  resp.\ coinvariant homology  relative to  the monodromy over $\p B$. \end{proposition}

\begin{proof}
From Proposition \ref{prop:MWcohomology} we learn that the sequence \eqref{eqn:MWsequence} can be identified with a sequence
\[
H^1(B,\p B; R^1\pi_*\ZZ)\to H^1(B; R^1\pi_*\ZZ)\to H^1(\p B; R^1\pi_*\ZZ)
\]
Each of these terms  appears in    Leray spectral sequences. For example, 
 \[
 E^{p,q}_2=H^p(B,\p B; R^q\pi_*\ZZ)\Rightarrow H^{p+q}(X,\p X)\cong H_{4-p-q}(X),
 \]
where the last isomorphism is Alexander duality.  Since all  fibers are integral, 
 this sequence degenerates, for the term $E^{p,q}_2$ vanishes unless $p\in\{1,2\}$ and $q\in\{0,1,2\}$ (a reducible fiber would
 give nonzero elements of $H^0_c (\mathring{B}, R^2\pi_*\ZZ)$). The isomorphism  $H^2(B, \p B; \pi_*\ZZ)=H^2(B, \p B)\cong H_0(B)\cong \ZZ$, then yields a short exact sequence
 \[
 0\to \ZZ\to H_2(X)\to \MW(\pi, \p\pi)\to 0.
 \]
 with $1\in \ZZ$ mapping to the fiber class  $e\in H_2(X)$. It follows that  $\MW(\pi, \p\pi)\cong H_2(X)/\ZZ e$.
 A similar argument identifies $H^1(B; R^1\pi_*\ZZ)$ with $H^2(X)^e$ and $H^1(\p B; R^1\pi_*\ZZ)$ with $H^2(\p X)^e$.

Now assume $H_1(X)=0$ and $\p B $ is connected. Then 
$H_3(X, \p X)=0\cong  H^1(X)=0$. The exact homology sequence for the pair $(X,\p X)$ plus 
 Alexander/Poincar\'e duality yields the exact sequence 
 \[
0\to H_2(\p X)\to H_2(X)\to H^2(X)\to H^2(\p X)\to 0.
\]
This sequence is self-dual and induces the exact sequence 
 \[
0\to H^2(\p X)/\ZZ e\to H_2(X)/\ZZ e\to H^2(X)^e\to H^2(\p X)^e\to 0.
\]
The Wang sequence for a fiber bundle over a circle  identifies 
$H^2(\p X)/\ZZ e$ with $H^1(X_b)^{\p B}$ and $H^2(\p X)^e$ with $H_1(X_b)_{\p B}$.
\end{proof} 

\begin{remark}\label{rem:recipe}
If we are in the special case of  Proposition \ref{prop:MWpair} (where $H_1(X)=0$ and $\p B$ connected), then 
that proposition shows that $\MW(\p \pi)\cong H_1(X_b)_{\p B}$ and that there is an extension of groups
\[
1\to H^1(X_b)^{\p B}\to \MW(\pi, \p\pi)\to \im (H_2(X)\to H^2(X))\to 1.
\]
We shall make the isomorphism  geometrically explicit in Lemma \ref{lemma:MWcircle}.  For now we explain how the embedding of   $H^1(X_b)^{\p B}$ in $\MW(\pi, \p\pi)$  is geometrically defined.

Let $U$ be a collar neighborhood of $\p B$ that does not contain any critical value  and identify it with $[0,1)\times \p B$. Choose  
a retraction $r: X_U\to X_{\p B}$ such that $(\pi_U, r): X_U\to U\times X_{\p B}$ is a diffeomorphism and let $\varphi: [0,1]\to \RR$ be a smooth function that is $0$ near $0$ and $1$ near $1$. We think of  $a\in H_1(X_b)^{\p B}$ as 
a locally constant family of classes $\{a_s\in H_1(X_s)\}_{s\in \p B}$ with $a=a_b$. Then a fiberwise translation for $\pi$  with support in $U$ is defined by specifying its restriction over $U$ to be
\[
(t,x)\in [0,1)\times X_s\mapsto (t, x+\varphi(t)a_s)\in [0,1)\times  X_s, \text{ where   $s\in \p B$.}
\]
in terms of the above  identifications.
This defines an element of $\MW(\pi, \p\pi)$ and up to a sign (which we did not bother to determine) this equals the image of $a$. 
\end{remark}

\subsection{Torus fibrations over a circle} 
We found in Proposition \ref{prop:MWpair} that if $\p B$  is connected (and hence a copy of the circle) and $b\in \p B$,  then $\MW(\p \pi)$ is the group of co-invariants of the monodromy action on $H_1(X_b)$.  As promised, we shall make this a bit more explicit (and geometric). Let us state and prove  this in the appropriate generality, as there is no extra cost in doing so.

\begin{lemma}\label{lemma:MWcircle}
Let $f: Y\to T$ be a locally trivial  fiber bundle over the unit circle  $T\subset \CC$ whose fibers are torsors of  affine $n$-tori (so that its structural group is $\GL(n, \ZZ)\ltimes \RR^n/\ZZ^n$). Let $A\in\GL(H_1(Y_1))\cong \GL(n, \ZZ)$ be the monodromy of this bundle. Then its Mordell-Weil group $\MW (f)$ is naturally isomorphic to the group $H_1(Y_1)_A$ of monodromy  co-invariants and this group acts simply transitively on $\pi_0(\G(f))$. 
\end{lemma}
\begin{proof}
Since the fiber $Y_1$ is connected, $f$ will certainly admit a section. Let $\sigma_0$ be one such. To give a fiberwise translation is then to give a section and vice versa and so this identifies $\MW (f)$ with the connected component group of the space of sections of $f$.

The section $\sigma_0$ identifies
$Y$ with the bundle whose fiber over $s\in T$ is $H_1(Y_s;\RR)/H_1(Y_s;\ZZ)$  (we may regard this as the Jacobian bundle)
and  we  can reconstruct $f$ from $A$ by taking the product $[0,2\pi]\times H_1(Y_1;\RR)$ and  dividing out by $H_1(Y_1;\ZZ)$ acting on $H_1(Y_1;\RR)$ by translations and identifying  $(0, x)\in\RR\times H_1(Y_1;\RR)$ with  $(2\pi, Ax)$ so that the orbit space  then projects onto $\RR/2\pi\ZZ\cong T$.  Any $v\in H_1(Y_1;\ZZ)$ determines a section $\sigma_v$ of $f$ by assigning to $e^{\sqrt{-1}t}$ the image of 
$(t, (2\pi)^{-1}t v)\in \RR\times H_1(Y_1;\RR)$  in $Y$. We leave it to the reader to check that $v\in  H_1(Y_1;\ZZ) \mapsto [\sigma_v]\in \MW(f)$ is onto and that $[\sigma_v]=0$ if and only if $v=A(v')-v'$ for some $v'\in H_1(Y_1;\ZZ)$. 
\end{proof}

\subsection{Mordell-Weil groups of fiberwise connected sums}\label{subsection:fibsum}
In this section we describe a connected sum construction for genus one fibered 4-manifolds  and determine its effect on Mordell-Weil groups. We shall illustrate this by the glueing of two rational elliptic fibrations producing an elliptic K3 manifold.

Assume that we are in the setting of  Proposition \ref{prop:MWcohomologyII}, where we have a genus one fibration  $\pi: X\to B$ with $B$ a copy of $\PP^1$ and $X$ simply connected (a copy of some $M_d$). Let $b\in B$ be a regular value of $\pi$. The \emph{real oriented blowup} of $b\in B$ produces a surface with boundary $\hat B_b$ diffeomorphic with a closed disk and a smooth map $f: \hat B_b\to \PP^1$ that is a diffeomorphism over $B\ssm \{b\}$ and where 
$\p\hat B_b:=f^{-1}b$  is the circle of rays in the tangent plane $T_b B$ (to be thought of as a boundary of $B\ssm\{b\}$). It is clear that the pull-back $\hat\pi: \hat X:=f^*X\to \hat  B_b$ of $\pi$ is a genus one fibration whose restriction to $f^{-1}b$ is the projection $\p\hat  B_b\times X_b\to \p\hat B$.

Suppose we are given another such pair 
$(\pi': X'\to B', b'\in B')$. Then the choice of  an affine isomorphism $u: X_b\cong X'_{b'}$ and an orientation \emph{reversing} isomorphism $v:T_bB\cong T_{b'}B'$ allows us to glue the two fibrations  
$\hat X\to \hat B$ and  $\hat X'\to \hat B'$ to produce a genus one fibration that we might denote by 
\[
(\pi,b)\#_{(u,v)} (\pi',b'): (X,b)\#_{(u,v)}(X',b')\to (B,b)\#_v (B',b'), 
\]
were we to insist on a notation that  expresses all the dependencies.
We will however simply write $\pi'': X''\to B''$.
Since the possible choices for $b$, $b'$ and $v$ vary in a connected family, the diffeomorphism type of $\pi''$ only depends on  the isotopy class of $u$. This connected sum is of the same type as its terms: $ B''$ is a copy of $\PP^1$ and $X''$ is simply connected. The discriminant of $\pi''$  appears here as the disjoint union of discriminants of $\pi$ and $\pi'$. It follows that the
arithmetic genus of $X''$ is the sum of the arithmetic genera of $X$ and $X'$. This construction, to which we shall refer  as a \emph{fiber connected sum}, was essentially introduced by Moishezon \cite{M}. 
The base $B''$ contains copies $\hat B$ and  $\hat B'$ over which this bundle reproduces  resp.\  $\hat\pi$ and  $\hat\pi'$ and these copies intersect in a copy $S$ of the circle that we shall identify with $\p\hat B$. 
The following theorem tells us how   $\MW(\pi'')$ is related to $\MW(\pi)$ and $\MW(\pi')$. 

\begin{theorem}\label{thm:fiberconnectedsum}
The  restriction of $\pi''$ to $S$  determines a  primitive  isotropic sublattice $I\subset \MW(\pi'')$ of rank $2$. We then have natural isomorphisms:
\[
I\cong H^1(X_b), \quad I^\perp/I\cong \MW(\pi)\oplus \MW(\pi'), \quad  \MW(\pi'')/I^\perp\cong I^\vee\cong H_1(X_b).
\]
with the middle isomorphism being an isomorphism of lattices.
\end{theorem}
\begin{proof}
The cohomology exact sequence of the sheaf $R^1\pi_*\ZZ$ for the pair $(B,\{b\})$ gives the short exact sequence
\[
0\to H^1(X_b)\xrightarrow{\delta}  H^1(B, \{b\}; R^1\pi_*\ZZ)\to H^1(B, R^1\pi_*\ZZ)\to 0 
\]
and also shows that $H^0(B, \{b\}; R^1\pi_*\ZZ)\cong H^0(B, R^1\pi_*\ZZ)=0$.  By Proposition \ref{prop:MWcohomology}, we can identify this with an exact sequence
\begin{equation}\label{eqn:deltasequence}
0\to H^1(X_b)\to  \MW (\pi, \{b\})\to \MW(\pi)\to 0 
\end{equation}
There is a similar exact sequence for $\pi'$. 
Next we consider a piece of  cohomology exact sequence of the sheaf $R^1\pi''_*\ZZ$ for the  pair $(B'', S)$:
\begin{multline}\label{eqn:MVforMW}
H^0(B'',S; R^1\pi''_*\ZZ)\to H^{0}(S; R^1\pi''_{S*}\ZZ)\xrightarrow{\delta''} H^1(B'',S; R^1\pi''_*\ZZ)\to\\
\to H^1(B''; R^1\pi''_*\ZZ)\to  H^1(S; R^1\pi''_*\ZZ)\to\cdots
\end{multline}
The bundle $\pi''_S$ is trivial, so $H^{0}(S; R^1\pi''_{S*}\ZZ)\cong H^1(X_b)$ and $H^1(S; R^1\pi''_{S*}\ZZ)\cong H^1(S)\otimes H^1(X_b)$. 
The group $H^q(B'',S; R^1\pi''_*\ZZ)$ decomposes as $H^q(B,\{b\}; R^q\pi_*\ZZ)\oplus 
H^q(B',\{b'\}'; R^1\pi'_*\ZZ)$.  So for $q=0$ this group is trivial  and for $q=1$,   $\delta''$ becomes  $(\delta, -\delta')$, where $\delta$ resp.\ $\delta'$ is  the coboundary 
in the short exact sequence for the pair $(B,\{b\})$ resp.\  $(B',b)$.   If we make these substitutions  in the sequence \eqref{eqn:MVforMW} and if we apply Proposition \ref{prop:MWcohomology}, then find the exact sequence
\[
0\to H^1(X_b)\xrightarrow{(\delta, -\delta')} \MW(\pi, \p\pi)\oplus \MW(\pi', \p\pi')\to \MW(\pi'') \to H^1(S)\otimes H^1(X_b)\to\cdots  
\]
The exact sequence \eqref{eqn:deltasequence} shows  that the   first  component of $(\delta, \delta')$ embeds  
$H^1(X_b)$ in  the summand $\MW(\pi, \p\pi)$ with cokernel  $\MW(\pi)$.
Since this image has trivial intersection with the image of  $(\delta, -\delta')$, it will embed $H^1(X_b)$ in $\MW(\pi'')$. The image of this embedding, which  will  be our  $I$,   
is isotropic for geometric reasons: if we identify  $\MW(\pi'')$ with a subquotient of $H^2(X'')$, then the image of $H^1(X_b)$ is represented by the image of the composite 
\[H^1(X_b)\cong H_1(X_b)\cong H_1(X_b)\otimes H_1(S)\xrightarrow{\times} H_2(X''_S)\to H_2(X'')\cong H^2(X''),\]
which also lies in the perp of the fiber class.  It is also clear that if we exchange the role of $\delta$ and $\delta'$, we get the same image.
The theorem now follows with a little diagram chase.
\end{proof}

An immediate corollary is the following.

\begin{corollary}\label{cor:fiberconnectedsum}
As a lattice, $\MW(\pi'')$ is isometric to $\MW(\pi)\oplus \mathbf{U}^{\oplus 2}\oplus \MW(\pi')$.
\end{corollary}

Applying the gluing formula of Corollary \ref{cor:fiberconnectedsum}, we are now able to compute the Mordell-Weil group of the generic elliptic fibration $\pi_d$.

\begin{proof}[Proof of Theorem \ref{theorem:computeMW}]
\label{proof:MWcompute}
The generic genus one fibration $\pi_d$ is (up to isomorphism) 
obtained by  starting out with $\pi_1$ and 
iterating the above connected sum construction. Since $\MW(\pi_1)\cong \Eb_8(-1)$, such a connected sum decomposition gives 
via Corollary \ref{cor:fiberconnectedsum} a corresponding decomposition of smooth Mordell-Weil lattices :
\[
\MW(\pi_d)\cong \Lambda_d(e)\cong  \Eb_8(-1)\perp \underbrace{2\Ub\perp \Eb_8(-1)\perp \cdots  \perp 2\Ub\perp\Eb_8(-1)}_{d-1}, 
\]
so with $d$ copies of $ \Eb_8(-1)$ and $(d-1)$ copies of $2\Ub$; here the first isomorphism comes from Proposition   \ref{prop:MWcohomologyII}.
\end{proof}

\section{Fibrations over a disk with two nodal fibers}
\label{section:2node}
We here study the most basic situation, namely a genus one fibration over a closed disk $B$ which has only two (interior) singular fibers, both of which are nodal. The special case where the two vanishing cycles on the fibers have intersection number $1$, and so the monodromy around $\partial B$ has  order 6, was studied by Gompf in \cite{G}.

We begin with looking at the  general case, setting up notation as we proceed, and then focus on the case when the two nodal fibers define the same vanishing cycle. Its topology, and in particular the group of  mapping classes it gives rise to,  turns out to be remarkably rich and subtle in structure. 

So we are given a genus one fibration $\pi :X\to B$ over a closed disk 
with only two singular fibers, both of which are nodal and lie over interior points of $B$. Without loss of generality we shall assume that 
$B$ is a closed complex disk of radius $r>1$ centered at $0$ and that the singular fibers lie over $1$ and $-1$.

\subsection{First properties}
It is not difficult to see that $X_{[-1,1]}:=\pi^{-1}[-1,1]$ is a deformation retract of $X$. So the homotopy type of $\pi$ is that of its restriction $X_{[-1,1]}\to [-1,1]$ over $[-1,1]$.  We shall describe  a simple geometric model  for this restriction.

As before, we let $T$ stand for the unit circle in $\CC^\times$, regarded an an oriented abelian Lie group. An orientation-preserving affine 
diffeomorphism $h_0$ of  $T^2$ onto the smooth fiber $X_0$ determines circle subgroups 
$T_{\pm}\subset T^2$ such that  if $E$ is the quotient space of $[-1,1]\times T^2$ obtained by collapsing 
$\{-1\}\times T_{-1}$ and  $\{1\}\times T_{1}$, then $h_0$ extends to a homeomorphism $h:E\cong X_{[-1,1]}$ over $[-1,1]$ 
that is a differentiable isomorphism of affine bundles over $(-1,1)$. 

It is then clear that for each sign $\pm$, the space $X_\pm:=\pi^{-1}[0, \pm 1]$  has  the fiber $X_{\pm 1}$ as a deformation retract and that the natural map 
$H_1(X_0)\to H_1(X_\pm)\cong H_1(X_{\pm 1})\cong \ZZ$ is a surjection. Its kernel $V_\pm\subset H_1(X_0)$, which via $h$ corresponds to the image of $H_1(T_\pm)\to H_1(T^2)$, is
the \emph{vanishing homology} of this degeneration.  See Figure \ref{figure:vanish1}.

\begin{figure}[h]
\includegraphics[scale=0.13]{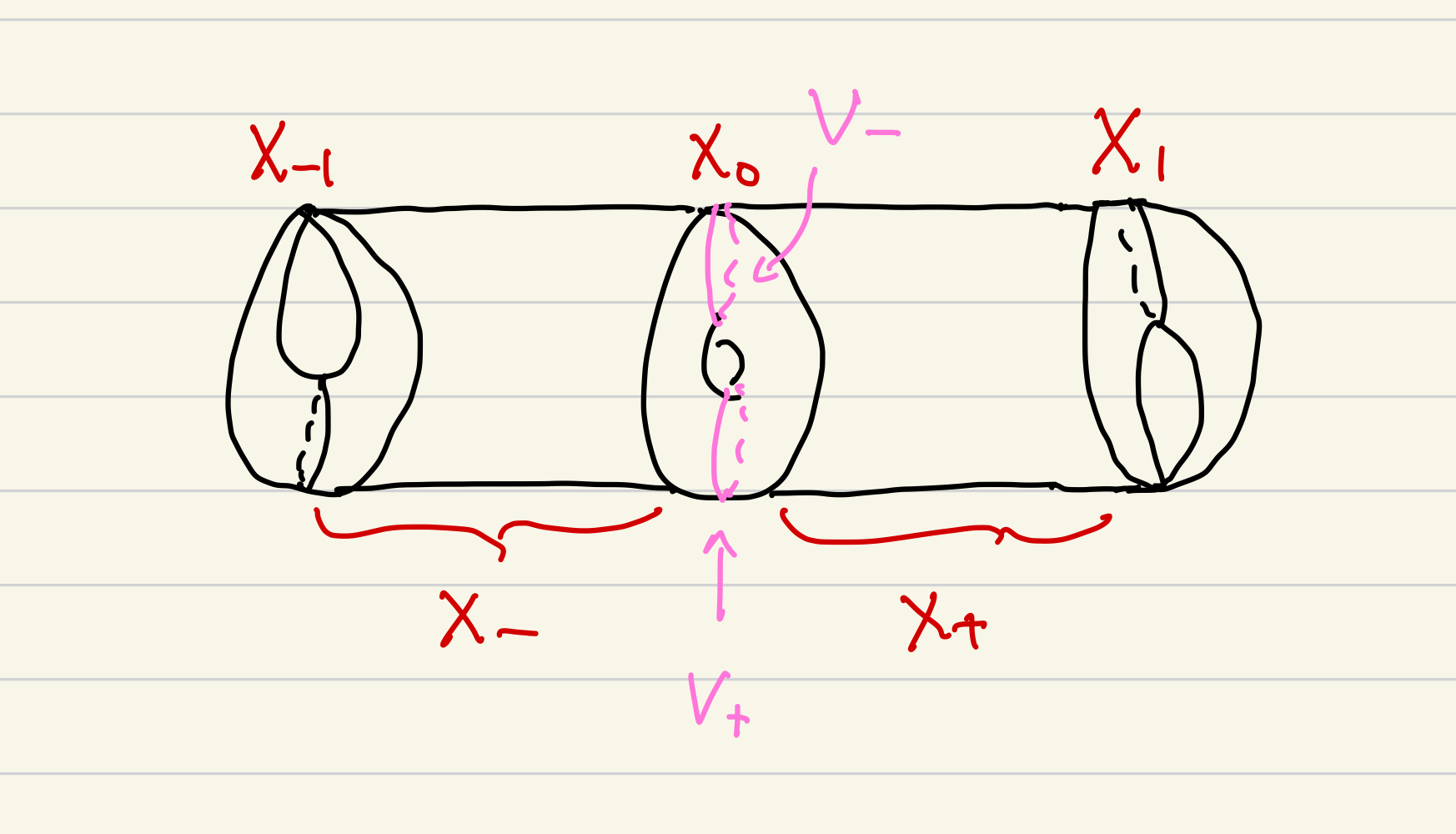}
\caption{\footnotesize The manifold $X_{[-1,1]}$.  In this case the vanishing homology subspaces $V_\pm\subset H_1(X_0)$ are generated by homologous classes.}
\label{figure:vanish1}
\end{figure}

The inclusion $X_0\subset X_\pm$  is up to homotopy obtained by the $2$-cellular extension of $T^2$ obtained by `filling in' the  circle subgroup $ T_{\pm}$ with a $2$-disk. If we fill in both disks we recover the inclusion  $X_0\subset X_{[-1,1]}\sim X$ up to homotopy.
 It is then  clear that the map $H_1(X_0)\to H_1(X_{[-1,1]})\cong H_1(X)$ is onto with kernel $V_-+V_+$ and that the map
 $H_2(X_0)\to H_2(X)$ is injective with cokernel naturally  identified with $V_+\cap V_-$. The following two propositions explain why our interest is mainly in the special case when $V_-=V_+$.

\begin{proposition}\label{prop:dichotomy}

The Mordell-Weil group $\MW(\pi)$ is naturally isomorphic to the cyclic group $H_1(X)\cong H_1(X_0)/(V_++V_-)$. So that group 
is infinite if and only if $ V_-=V_+$, and is  trivial  if and only if 
$V_-+V_+=H_1(X_0)$. 
\end{proposition}
\begin{proof}The inclusion $[-1,1]\subset B$ induces an isomorphism $\MW(\pi)\cong H^1(B; R^1\pi_*\ZZ)\cong H^1([-1,1]; R^1\pi_{*}\ZZ)$.
The restriction of   $R^1\pi_{*}\ZZ$ to $[-1,1]$ is  simple: over $(-1,1)$ it is constant equal to $H^1(X_0)$ and its stalk in  $\pm 1$ is $H^1(X_{\pm 1})$. So $H^1([-1,1]; R^1\pi_{*}\ZZ)$ is the cokernel of the natural map 
 $H^1(X_{1})\oplus H^1(X_{-1})\to H^1(X_0)$.  Alexander duality turns this  into  $V_+\oplus V_-\to H_1(X_0)$ defined by the inclusions $V_\pm\subset H_1(X_0)$.
\end{proof}

In case $V_-$ and $V_+$ are distinct  (so that $\MW(\pi)\cong H_1(X)$ is finite), we can say a bit more:

\begin{proposition}\label{lemma:less interesting case}
When $V_-$ and $V_+$ are distinct, then $\MW(\pi, \p\pi)$ is trivial. The  natural map $\MW(\pi)\to \MW(\p\pi)$ is (therefore) injective with cokernel naturally identified with $H_1(X)$.

Furthermore, if  $q$ is the order of $H_1(X)\cong H_1(X_0)/(V_-+V_+)$, then the trace of the monodromy along $\p B$ on  $H_1$ of a fiber equals $2-q^2$. In particular, this monodromy is hyperbolic unless $q=1$ (i.e., when $V_-+V_+=H_1(X_0)$), in which case it has order $6$.
\end{proposition}
\begin{proof}
The Mayer-Vietoris sequence of the pair $(X_-, X_+)$ gives the exact sequence
\[
0\to H_2(X_0) \to H_2(X)\to H_1(X_0)\to H_1(X_-)\oplus H_1(X_+)\to H_1(X).
\]
The map  $H_1(X_0)\to H_1(X_\pm)$ is given by $\pm 1$ times the natural map $H_1(X_0)\to H_1(X_0)/V_\pm$. Since $V_-\not=V_+$, it follows that $H_1(X_0)\to H_1(X_-)\oplus H_1(X_+)$ is injective. We deduce that  $H_2(X)$ is generated by $e$ so that $\MW(\pi, \p\pi)$, which by  Proposition \ref{prop:MWpair}  can be identified with $H_2(X)/\ZZ e$, is trivial. This 
 proposition  also identifies $\MW(\pi)\to \MW(\p\pi)$ with $H^2(X)^e\to H^2(\p X)^e$. The latter  map fits in the exact sequence (derived from that of the pair $(X, \p X)$)
\[
 H^2(X)^e\to H^2(\p X)^e\to H^3(X, \p X)\to H^3(X).
 \]
Since $X$ has the homotopy type of a $2$-dimensional CW complex,  $H^3(X)=0$ and $H^3(X, \p X)\cong H_1(X)$ by Alexander duality. This proves the first assertion.

For the second statement, we choose a basis $e_1,e_2$ of $H_1(X_0)$ such that $V_+$ is spanned by $\delta_+=e_1$ and $V_-$ by $\delta_-=pe_1+qe_2$ (where $p$ and $q$ must be relatively prime). Then the monodromy is the composite of two Picard-Lefschetz transformations, namely  $x\mapsto x + (x\cdot \delta_+)\delta_+$ followed by $x\mapsto x + (x\cdot \delta_-)\delta_-$.
A  straightforward computation shows that the trace of this composite is $2-q^2$, as asserted.  So for $|q|=1$, the trace is $1$ and any element of $\SL_2(\ZZ)$ with that property has order $6$, whereas 
for $|q|\ge 2$, the trace is $\le -3$ and any element of $\SL_2(\ZZ)$ with that property is hyperbolic.
\end{proof}

\begin{remark}\label{rem:nonunique}
The homeomorphism $h:E\cong X_{[-1,1]}$ is \emph{not} unique up to relative isotopy. This is because the group of
 isotopy classes of self-homemorphisms of $X_{[-1,1]}$ relative to the projection onto $[-1,1]$ can be nontrivial. Indeed that group is naturally isomorphic with $H_1(X)$.
To be precise, note that if $f: [-1,1]\to T^2$ is a smooth loop that is  constant equal to the identity element on  a neighborhood of $\{-1,1\}$, then we can postcompose $h$ with the 
the map induced by 
\[
(s,  {u})\in [-1,1]\times T^2\mapsto (s, f(s){u})\in [-1,1]\times T^2
\]
and this may represent a different isotopy class. Such an $f$ defines an element of $H_1(T^2)$ and it is  
not hard to check that the change in relative isotopy class only depends on that element. 
All the relative isotopy classes  of such $h$ are thus obtained. In fact,  $H_1(X_0)\cong H_1(T^2)$ permutes the choices transitively  and 
since $H_1(T_\pm)$ gets killed over $\pm 1$, that action will be  through the cokernel of 
$H_1(X_0)/(V_-+V_+)\cong H_1(X)$. On may verify that the  latter action is simply transitive. 
\end{remark}

\emph{From now on we assume that the two nodal   fibers of $\pi$ have the same vanishing homology, so that $V_-=V_+$.} 
\\

We assume that  if  $\delta,\delta'\in H_1(T^2)$ denote the generators defined by the factors, then $V_-=V_+=\ZZ\delta$. So then $H_1(X)\cong H_1(X_0)/\ZZ \delta$ and  $H_2(X_0)$ embeds in $H_2(X)$ with cokernel identified with $V_+$ (hence is free abelian of rank two).  We  can  exhibit a generator of that cokernel explicitly  by representing the two vanishing cycles in $T^2$  by 
the same circle $T\times\{1\}$: this defines a   $2$-sphere, namely the image of $[-1,1]\times (T\times\{1\})$ under the projection $[-1,1]\times T^2\to E$. This is indeed a topological oriented $2$-sphere  $C$ in $X_{[-1,1]}$. The fundamental classes of $C$ and of the central fiber $X_0$ make up a free basis of $H_2(X)$. We shall see in the next subsection that we can arrange $C$ is smooth with $C\cdot C=-2$.

The monodromy around both $1$ and $-1$ is given by 
\[
\tau_{\delta}: a\in H_1(T^2)\mapsto a+(a\cdot \delta)\delta \in H_1(T^2).
\]
Hence the  monodromy of $\pi$ over $\p B$  is $\tau_{\delta}^2$ (connect $0\in B$ with the point of $\p B$ via an interval on the  negative imaginary axis to identify a fiber over $\p B$ with $T^2$). By Lemma \ref{lemma:MWcircle} this identifies $\MW(\p\pi)$ with 
the cokernel of $\tau_{\delta}^2-1: a\mapsto 2(a\cdot \delta)\delta$, i.e., with $(\ZZ/2)\delta\oplus  \ZZ\delta'$.

The situation is quite different from that of Proposition \ref{lemma:less interesting case}:

\begin{proposition}\label{prop:equinodalMW}The exact sequence $\MW(\pi, \p\pi)\to \MW(\pi)\to \MW(\p\pi)$ is naturally isomorphic to the sequence 
\[
H_2(C) \xrightarrow{2\times} H_1(X_0)/V_+\to H_1(X_0)/ 2 V_+,
\]
where  the first map sends a generator of $H_2(C)$ to twice a generator of  $H_1(X_0)/V_+$, so that the image of the second map is the order two subgroup $V_+/2 V_+$.
\end{proposition}
\begin{proof}
We already established in Proposition \ref{prop:dichotomy} and  Proposition \ref{lemma:less interesting case} isomorphisms $\MW(\pi)\cong H_1(X_0)/(V_++V_-)=H_1(X_0)/V_+$ and  $\MW(\p\pi)\cong H_1(X_0)/ 2 V_+$. For the rest of the argument, we recall that the $\MW$-exact sequence in question appears as part of a long exact sequence
\begin{multline*}
\cdots\to H^0(B, \p B;R^1\pi_*\ZZ)\to H^1(B, \p B;R^1\pi_*\ZZ)\to H^1(B;R^1\pi_*\ZZ)\to H^1(\p B;R^1\pi_*\ZZ)\to\\
\to  H^2(B, \p B;R^1\pi_*\ZZ)\to H^2(B;R^1\pi_*\ZZ)\to\cdots
\end{multline*}
Since $H^0(B, \p B;R^1\pi_*\ZZ)$ consists of the  sections of $R^1\pi_*\ZZ$ that vanish on $\p B$, this group is trivial.
The group  $H^2(B;R^1\pi_*\ZZ)$ is also trivial, because  it restricts isomorphically to $H^2([-1,1];R^1\pi_*\ZZ)$ and that group is zero too, so that there is an exact sequence
\begin{center}
\begin{small}
\begin{tikzcd}
0\to  H^1(B, \p B;R^1\pi_*\ZZ)\arrow[r]\arrow[d, equal] &H^1(B;R^1\pi_*\ZZ)\arrow[r]\arrow[d, equal]  &H^1(\p B;R^1\pi_*\ZZ)\arrow[r]\arrow[d, equal]  &H^2(B, \p B;R^1\pi_*\ZZ)\to 0\\
\MW(\pi, \p\pi)  & \MW(\pi) &\MW(\p\pi) &
\end{tikzcd}
\end{small}
\end{center}
We are going to determine the as yet unknown $H^1(B, \p B;R^1\pi_*\ZZ)$ and $H^2(B, \p B;R^1\pi_*\ZZ)$.
For this we use the Leray spectral sequence 
\[
E_2^{p,q}=H^p(B, \p B; R^q\pi_*\ZZ)\Rightarrow H^{p+q}(X, \p X)\cong H_{4-p-q}(X),
\]
where we keep in mind that $H_1(X)\cong H_1(X_0)/V_+$, that  $H_2(X_0)$ embeds in $H_2(X)$ with cokernel  
$H_2(C)$  and $H^k(X)=0$ for $k>2$.

Note that  $E^{p,q}$ vanishes  unless $0\le p, q\le 2$. We already observed  that $E_2^{0,1}=0$, so that  $E_2^{2,0}=
H^2(B, \p B; R^0\pi_*\ZZ)=H^2(B, \p B; \ZZ)\cong \ZZ$ must embed in $H^2(X)$. Also,  $E_2^{0,2}=
H^0(B, \p B; R^2\pi_*\ZZ)\cong H^0(B, \p B; \ZZ)=0$ and hence the cokernel  of $H^2(B, \p B; \ZZ)\to H^2(X, \p X; \ZZ)$ 
can be identified with $E^{1,1}=H^1(B, \p B; R^1\pi_*\ZZ)$. Duality identifies  $\pi^*: H^2(B, \p B; \ZZ)\to H^2(X, \p X; \ZZ)$ with
the Gysin map $\pi^!: H_0(B)\to H_2(X)$, which assigns to $1\in H_0(B)$ the fiber class. In other words, the  cokernel in question is identified with that of $H_2(X_0)\to H_2(X)$, i.e., with $H_2(C)$.

This also shows that $E_2^{2,1}=H^2(B, \p B;R^1\pi_*\ZZ)$ is isomorphic to $H^3(X, \p X)\cong H_1(X)\cong V_+$. So the exact sequence above becomes
\[
0\to H_2(C)\to H_1(X_0)/V_+\to H_1(X_0)/ 2 V_+\to V_+\to 0
\]
and hence the maps in it must have the asserted properties. 
\end{proof}

\subsection{A Weierstra\ss\ model and a circle action}\label{subsect:weierstrass}
A  Weierstra\ss\ model for an equinodal pair is given by the family  of plane cubics defined by
\[
y^2+x^3+x^2 +\eps  (s^2-1), 
\] 
where we let $\eps>0 $ be small, and where we must add the section at infinity in order to  make this a (degenerating) family of elliptic curves. 

This sphere can be made explicit in terms of our Weierstra\ss\ model: for $s\in (-1,1)$, the
cubic equation $x^3+x^2 +\eps  (s^2-1)=0$ has three distinct real roots: one close to $-1$ and two close to zero and 
the real locus of the cubic curve  $y^2+x^3+x^2 =\eps  (1-s^2)$ in $\RR^2$ has two connected components, one of which is compact (so a differentiable circle) which meets the $x$-axis in the two roots close to zero  and the other is a properly embedded copy of $\RR$  that is compactified by the point at  infinity. If we let $s\to \pm 1$, then the compact component  shrinks to a point, so that if we let $s$ run over the interval $[-1,1]$, then these compact components trace out a $2$-sphere. This
$2$-sphere, which is easily checked to be smooth,  is our $C$. This $C$ in turn, appears as a vanishing 2-cycle of the  following degeneration of cubic surfaces 
\[
t=\big(y^2+x^3+x^2) +\eps  s^2,\quad  |t|\le \eps, 
\]
which for $t=0$  acquires a quadratic singularity at $(x,y,s)=(0,0,0)$. A (Picard-Lefschetz) vanishing cycle in dimension two has self-intersection $-2$, and so it follows that $C\cdot C=-2$.

\begin{remark}\label{rem:}
One can check that the  homology classes $c\in H_2(X)$ of $C$ and $e\in H_2(X)$ of a smooth fiber make up a free basis of $H_2(X)$.
Their intersection numbers are $c\cdot c=-2$, $e\cdot e=0$ and $c\cdot e=0$. In particular, $e$ spans the radical of the intersection pairing on $H_2(X)$.
\end{remark}

\begin{remark}\label{rem:}
It is worth noting that the  the group $\Trans(\pi)$ of fiberwise translations  of $\pi$ contains as a subgroup a copy $T_\pi$ of the circle group $T$. It fixes the nodes in the singular fibers, but is otherwise free and we can take the $2$-sphere $C$ to be such that each fiber of the projection $\pi|C: C\to [-1,1]$ is a $T_\pi$-orbit. This action commutes with the geometric monodromy  and its orbit space defines a $T$-bundle over $B$ (which must be trivial since $B$ is contractible). A trivialization  $X\to T$ of this bundle  induces an isomorphism $H_1(X)\to H_1(T)\cong\ZZ$.
\end{remark}

\subsection{The braid action} Since the embedded smooth sphere $C$ has self-intersection number $-2$,  it determines a $2$-dimensional Dehn twist $T(C)$ as an element of  the mapping class group  of the pair $(X, \p X)$. It is well-known that  $T(C)$ only depends on the isotopy class of $C$ and has square isotopic to the identity. We  will  represent  $T(C)$ as the monodromy of the function $t=\big(y^2+x^3+x^2) +\eps  s^2$ above and this will in fact produce a relative mapping class $\tau (C)$ in $\Mod(\pi, \p\pi)$ that lifts  $T(C)\in \Mod(X, \p X)$. 

Recall that the  connected component group of the group diffeomorphisms of $B$  that are the identity near $\p B$ and preserve $\{-1,1\}$ is infinite cyclic with a distinguished (positive) generator. We can  represent this generator by the simple braid, i.e.,  as  the 
monodromy $\tau$ of the  bundle pair over the circle $(T\times B, \Dscr)\to T$, where  $\Dscr$ is the set of $(u,s)\in T\times B$ satisfying $s^2=u$. Its square  is the Dehn twist along a $\p B$ and  generates the mapping class group of the pair 
$(B, \p B\cup \{-1,1\})$.
We will lift this to a map $T\times X\to T\times B$ over $T$  (making it a family of genus one fibrations of $X$ parametrized by $T$) in such a manner that $\Dscr$ 
is its set of critical values  and the fiber over $1\in T$ is the given $\pi: X\to B$. Although it ought to be possible to express this entirely in 
differential-topological terms, it is much easier to exploit the complex structure and rely on a bit of singularity theory. To this end, we return to the 2-parameter family of cubic curves that we used to define the vanishing sphere $C$, except that we rescale the $t$-parameter:
 \[
y^2+x^3+x^2 =\eps t-\eps  s^2. 
\]
We here let $s$ run over $B$ and let $t$ run over an open complex  disk $\Delta$ of radius $>1$ (as before, we assume  $\eps >0$ small). Write $\Xcal$  for the preimage 
of $B\times \Delta$, including its section at infinity and regard $\Xcal\to \Delta$ as a family of elliptic fibrations
$\{\pi_t:\Xcal_t\to B\}_{t\in \Delta}$, which for $t=1$ returns $\pi :X\to B$.
The singular fibers are defined by $t=s^2$, so if $t$ runs over the circle of radius $1$ in $\Delta$ the critical values of $\pi_t$ traverse a basic braid $\Dscr$. 

Use $(x,y,s)$ as global coordinates for $\Xcal$ minus its section at infinity, so that $t$ is there given by
$t=\eps ^{-1}\big(y^2+x^3+x^2) + s^2$.
All fibers of this map are smooth surfaces except the  central fiber $\Xcal_0$,  which  has an ordinary double point.
As is well-known, the monodromy of this family, regarded as a diffeomorphism of $\Xcal_1 =X$ onto itself,  is a $2$-dimensional Dehn twist defined by the vanishing 2-sphere $C$. Since the  action $T_\pi$ on $X$ extends  
to one on $\Xcal$, we find:

\begin{corollary}\label{cor:2Dehn}
Let $\tau$ be a diffeomorphism of $B$ that is the identity near $\p B$,  exchanges $1$ and $-1$ and represents the simple braid. Then $C$ determines a lift of $\tau$  to an automorphism of $\pi$ that commutes with the action of $T_\pi$.
Its image $\tau(C)$ in $\Mod(\pi, \p\pi)$ is of infinite order, but  its image in $\Mod(X, \p X)$
is the $2$-dimensional Dehn twist $T(C)$ and hence of order two. 
\end{corollary}

\subsection{Enumerating the classes of sections}
According to Proposition \ref{prop:dichotomy}, $\MW(\pi)$ is naturally isomorphic with the infinite cyclic group
$H_1(X_0)/\ZZ\delta\cong \ZZ\delta'$

 Since the set $\pi_0(\G(\pi))$ is a $\MW(\pi)$-torsor, we can enumerate its elements by the integers. We make this concrete as follows.
The image of a section  of  $\pi: X\to B$ in  $\pi_0(\G(\pi)$) is determined by its restriction to $[-1,1]$. Since a section  avoids the singular points, it will  (via $h$) define a section 
of the trivial bundle $[-1,1]\times T^2\to T^2$  with the property that its value in $\pm 1$ avoids $T_\pm=T\times \{\pm 1\}$. So this is given by a differentiable map $f=(f',f''): [-1,1]\to T\times T$ with $f''(\pm 1)\not=1$. 
All that matters for the class of this section is the connected component  of $f$ in the space of such maps.  Since we do not impose a boundary condition on $f'$, it is in fact only the connected component of $f''$ that counts. As $T\ssm \{1\}$ contains ${-1}$ as a deformation retract, we may just as well agree that $f''(\pm 1)=-1$. So such  sections are enumerated by the homotopy classes of loops in $T$ with base point $-1$. These homotopy classes are naturally indexed by the integers  $n \in \ZZ$. Indeed,  each such homotopy class contains a unique representative of the form  
\[
f''_n: s\in [-1,1]\mapsto -e^{n(s+1)\pi\sqrt{-1}}\in T.
\]
We denote by $\sigma_n$ a section of $\pi$ that extends the section over  $[-1,1]$ defined by one for which the second component equals $f''_n$. In order that $\sigma_n$ be smooth some precaution is necessary: we should 
replace the multiplication factor $e^{n(s+1)\pi\sqrt{-1}}$ by $e^{n\varphi (s)\pi\sqrt{-1}}$, where $\varphi :[-1,1]\to [0,2]$ is constant $0$ near $-1$ and constant $2$ near $1$.  Observe that $\sigma_0$ does not meet $C$,  because 
via the homeomorphism  $X_{[-1,1]}\cong E$,  the second $T$-coordinate  is constant $-1$ on $\sigma_0$ and constant $+1$ on $C$. So if we write $[\sigma_n]$ for the image of 
$\sigma_n$ in $H_2(X, \p X)\cong H^2(X)$, then $\la \sigma_0 | c\ra =0$. 

We identified $\MW(\p\pi)$ with  $(\ZZ/2)\delta\oplus  \ZZ\delta'$ and so $\MW(\p\pi)$ has a unique element of order two. 
Proposition \ref{prop:equinodalMW} implies:

\begin{proposition}\label{prop:sectionrestriction}
The image of  $\sigma_n-\sigma_0$ in $\MW(\p\pi)$ is torsion and only depends on the parity of $n$: it is zero or of order two depending on whether $n$ is even or odd.
\end{proposition}

The preceding suggests that we consider the diffeomorphism of $[-1,1]\times T^2$ over $[-1,1]$ onto itself defined by
\[
F: (s; u_1,u_2)\mapsto (s; u_1, u_2 e^{\varphi (s)\pi\sqrt{-1}}),
\]
where $\varphi :[-1,1]\to \RR$ is as before: constant $0$ near $-1$ and constant $2$ near $1$.
As this is the identity over neighborhood of $\{-1,1\}$, it induces a homeomorphism of $E$ onto itself
that extends as a smooth fiberwise translation of $X$. We denote that diffeomorphism by $\Phi$.

The identity $f''_n (s)e^{n\varphi (s)\pi\sqrt{-1}}=f''_{n +1}(s)$ implies that  $\Phi$ takes the class $\sigma_n$ to that of 
$\sigma_{n+1}$. Since the classes of  
$\sigma_n|\p B$ and  $\sigma_{n +1}|\p B$ are different, $\Phi|\p B$ will be nontrivial (of order two), but  $\Phi^2|\p B$ is in the identity component. We conclude:

\begin{proposition}\label{prop:}The Mordell-Weil group  $\MW(\pi)$ is infinite cyclic with a generator 
represented by $\Phi$; this generator takes the class of $\sigma_n$ to the class of  $\sigma_{n+1}$. 
The section $\sigma_0$ is disjoint from $C$.
The relative Mordell-Weil group  $\MW(\pi, \p \pi)$ is generated by the class of $\Phi^2$.
\end{proposition}

We put  $C_n:= \Phi^n (C)$. Since $\sigma_n=\Phi^n\sigma_0$, it follows that $C_n$ and $\sigma_n$ are disjoint.
The following corollary shows among other things that $\sigma_n$ is the only section with that property. 

\begin{corollary}\label{cor:}
The class of  $C_n= \Phi^n (C)$ is equal to  $c+ne$. In particular, $\sigma_m\cdot C_n=\sigma_{m-n}\cdot C=m-n$.
A basis of  $H_2(X, \p X)\cong H^2(X)$ is given by the classes of the sections $\sigma_0$ and $\sigma_1$.
\end{corollary}
\begin{proof}
By construction, $\Phi$ induces in $E$  the map defined by $F$. Recall that the oriented topological $2$-sphere 
$S$ is the image of $ [-1,1]\times T^1\times\{1\}$ in $E$. 
There is a projection of $E$ onto the quotient $\check{T^2}$ of $T^2$ that is obtained by collapsing  $T\times\{1\}$ to a point. Under this projection, $S$ is mapped to a point, but the central fiber $T^2$ maps  with degree one on 
$\check{T^2}$.  Since $F_*(S)$ also maps to $\check{T^2}$ with degree one, it follows that $\Phi_*(c)=c+e$. If we combine this with the fact that $\Phi_*$ fixes $e$, we find that  
$[\Phi^n(C)]=(\Phi_*)^n(c)=c+ne$.

We check the second assertion by evaluating $\sigma_0$  and $\sigma_1$ on the basis $(e,c)$ of 
$H_2(X)$ and prove that the resulting $2\times 2$-matrix has absolute determinant $1$: indeed,  
$\sigma_0 \cdot e=1$, $\sigma_0 \cdot c=0$ and $\sigma_1\cdot e=1$, 
$\sigma_1\cdot c=\Phi(\sigma_0)\cdot c=\sigma_0\cdot  \Phi_*^{-1}(c)=\sigma_0\cdot(c-e)=-1$. 
\end{proof}

Note the classes  $[C_n]=c+ne$ yield up to sign all elements in $H_2(X)$ of self-intersection number $-2$.
So the Mordell-Weil group $\MW(\pi)$ acts simply transitively  on this collection of isotopy classes of embedded
$2$-spheres in $X$ (when we regard them as unoriented submanifolds) as well as on the set  $\pi_0\G(\pi)$ of 
isotopy classes of sections. The group $\MW(\pi, \p \pi)$ has therefore two orbits in these sets. 

\begin{remark}[Relation to degenerations]
\label{rem:I2observation}
This decomposition into two orbits can be understood in terms a degeneration of $\pi$. For this purpose  it is convenient  to  represent the vanishing cycles $T_\pm$ in $T^2$ not by the same circle $T\times\{1\}$ as we did above, 
but by two isotopic ones, namely by taking  $T_\pm=T\times\{\pm 1\}$.  If we  then let the critical values 
$\{-1,1\}$ of $\pi$ then both move to $0$ (thereby shrinking the interval that connects them to a singleton), we   
obtain an elliptic fibration $\pi' : X'\to B$ whose only singular fiber  $X'_0$  is  topologically obtained from $T^2$ by contracting both $T\times \{1\}$ and  $T\times \{-1\}$. (This limiting process can  in fact be carried out in the holomorphic category;  the  central singular  fiber is then of Kodaira type $\Itwo$: it has two irreducible components,  each being a Riemann sphere with self-intersection number $-2$ 
and meeting the other transversally in two points.) 

The set of sections $\G (\pi')$ that meet a given connected component of the smooth part of $X'_0$ make up a connected set and so  $\pi_0\G (\pi')$ 
has only two elements. Indeed, $\MW(\pi')$ is of order two with the nontrivial element exchanging  these two.
The natural map $\MW(\pi')\to \MW(\p\pi')=\MW(\p \pi)$ is injective with image the torsion subgroup of $\MW(\p \pi)$ and $\MW(\pi', \p\pi')$ is trivial.

If we make  the above choices for $T_\pm$, then these 2-spheres also live in our $E$ 
as topological spheres  $S_\pm$, namely as the images of the two maps 
\[
(s,t)\in [-1,1]\times T^1\mapsto (\pm s, \pm e^{\pi\sqrt{-1}s}, t)\in [-1,1]\times T^2.
\]
When suitably oriented, the image of $S_+$ resp.\ $S_-$  in $X$ under the embedding $E\hookrightarrow X$ is 
isotopic  to $C$ resp.\ $-C_{-1}$ (whose class is $e-c$). The two $\MW(\pi, \p \pi)$-orbits in $\pi_0\G(\pi)$ 
degenerate into the two elements of  $\pi_0\G(\pi')$.
\end{remark}

In Corollary \ref{cor:2Dehn} we noted that the smoothly embedded 2-sphere $C$  (with self-intersection $-2$)
defines $\tau(C)\in  \Mod(\pi, \p\pi)$. It is clear that  
\[\tau (C_n)=\tau (\Phi^nC)= \Phi^n\tau(C)\Phi^{-n}.\] 
Since each $\tau(C_n )$ lifts the same  diffeomorphism of $B$,  any two such `differ' by  an element  of  $\MW(\pi, \p \pi)$.  In particular, $\tau(C_{1})\tau(C)^{-1}$  
defines an element  of $\MW(\pi, \p \pi)$ and hence represents an even power of $\Phi$.
\\

In what follows we make use of a ``variation homomorphism'', defined as 
follows.  Let $(X,Y)$ be a topological pair with $Y\subset X$ closed and $h:X\to X$ a homeomorphism  that is  the identity on $Y$. If $z$ is  a simplicial $k$-chain on $X$ such that $\p z$ has its support on  $Y$, then $h_*z-z$ is a  $k$-cycle on $X$. Thus $h_*-1$ induces what is called the \emph{variation map}
\[
\var (h): H_\pt (X, Y)\to  H_\pt(X).
\] 
It only depends on the relative homotopy class of $h$. One justification for this notion is that if $j :(X,Y)\subset (X',Y')$ is an embedding which 
induces an isomorphism on relative homology, then if $h': X'\to X'$ is the extension of $h$ to the identity, 
then $\var(h')$ is expressible in terms of $\var (h)$ as the composite
\[
\var (h'):  H_\pt (X', Y')\xleftarrow{j_* \cong} H_\pt(X, Y)\xrightarrow{\var (h)} H_\pt (X)\xrightarrow{j_*}H_\pt(X').
\]
We further note that $\var$ satisfies the  cocycle condition: if $h_1,h_2:(X,Y)\to  (X,Y)$ are as above,  then 
$\var (h_1h_2)=\var (h_1)h_{2*} +\var(h_2)$.

\begin{proposition}\label{prop:spheres and sections}
The image of $\tau(C_{1})\tau(C)^{-1}$ in $\MW(\pi, \p \pi)$ is that of $\Phi^2$  (which we recall defines the generator 
of $\MW(\pi, \p \pi)$) and its associated variation homomorphism equals
\[
x\in H_2(X, \p X)\mapsto \la x, e\ra c- \la x,c\ra e +\la x, e\ra e\in H_2(X), 
\]
hence comes from the Eichler transformation $E(e\wedge c)$. 
The class of $\sigma_{2m}-\sigma_0$ in $H_2(X)$ (considered as the image of 
$\var(\Phi^{2m})[\sigma_0]$)  is  $mc+m^2e$.
\end{proposition}
\begin{proof}
We first determine the variation homomorphism of $T(C_{1})^{-1}T(C)$. For $x\in H_2(X, \p X)$: 
\begin{multline*}
\var(T(C_{1})T(C)^{-1})(x)=\var(T(C_{1})T(C)_*(x)+ \var(T(C))(x)\\
=\la x+\la x, c\ra c, c+e\ra (c+e)+ \la x, c\ra c
=\la x, e\ra c- \la x,c\ra e +\la x, e\ra e.
\end{multline*}
This is indeed the variation of the Eichler transformation defined by $e\wedge c$.
This  shows in particular that $\tau(C_{1})\tau(C)^{-1}$ takes $c$ to $c+2e$. Since 
$\Phi^n_*$ takes $c$ to $c+ne$, it follows that $\tau(C_{1})\tau(C)^{-1}=\Phi^2$.

Recalling that  $\Phi^{2m}$ acts as the Eichler transformation defined by $me\wedge c$, the last assertion then follows from
\[
\var(\Phi^{2m})[\sigma_0]= \la \sigma_0, me\ra c- \la \sigma_0,c\ra me +\la \sigma_0, me\ra me= mc+m^2e. \qedhere
\]
\end{proof}

Remark \ref{rem:I2observation} above suggests the following alternate way of looking at 
$\Mod(\pi, \p\pi)$ and its map to $\Mod(X, \p X)$.

\begin{corollary}\label{cor:MWpresentation}
A presentation of  $\Mod(\pi, \p\pi)$  has as  generators  
$\tau(C)$ and $\tau(C_1)$ and  the relation 
$\tau(C)^2=\tau(C_1)^2$,  so that $\Mod(\pi, \p\pi)$ is a central extension
\[
1\to\mathbb{Z}\to \Mod(\pi, \p\pi)\to (\ZZ/2) * (\ZZ/2) \to 1,
\]
the center being generated by $\tau(C)^2$ (here $(\ZZ/2) *(\ZZ/2)$ is the infinite dihedral group).

The relative Mordell-Weil group $\MW(\pi, \p \pi)$  is the infinite cyclic subgroup of $\Mod(\pi, \p\pi)$ generated by  the basic translation $F(C):=\tau(C)^{-1}\tau(C_{1})$.  In particular, 
\[\Mod(\pi, \p\pi)\cong \MW(\pi, \p \pi)\rtimes \tau(C)^\ZZ\] 
with $\tau(C)$ acting by inversion: 
\[ 
\tau(C)F(C)\tau(C)^{-1}=F(C)^{-1}.
\]

The image of $\tau(C)$ in the relative mapping class group $\Mod(X, \p X)$ is the Dehn twist $T(C)$ and hence of order 2. Indeed, $\Mod(X, \p X)$ is  the quotient group of $\Mod(\pi, \p\pi)$ defined by this relation and hence  is the infinite dihedral group $\MW(\pi, \p \pi)\rtimes \ZZ/2$).
  \hfill $\square$
\end{corollary}

The groups $\Mod(\pi)$, $\Mod(\pi, \p\pi)$, $\Mod (X, \p X)$ that appear here are of course naturally associated to $\pi$, but $C$ is not privileged over any of  
the spheres $\pm C_n$ with $n\in \ZZ$. On the other hand,  as Proposition \ref{prop:spheres and sections} shows, with the  choice of $C$ comes the choice of a component of $\G(\pi)$ (containing a section disjoint with $C$): 
there is a natural bijection between the collection of (unoriented) spheres $\{C_n\}_{n\in \ZZ}$ and $\pi_0(\G(\pi))$.

\section{Small support realizations of elements of the Mordell-Weil group}
\label{section:smallsup}

Our goal in this section is to prove Theorem \ref{thm:smallsupport}, that each element of $e^\perp/\ZZ e$ has a representative in $\Diff(M)$ 
with small support, and of a very special form.  

\subsection{Equinodal arcs} We shall see that the `equinodal' fibration studied in \S\ref{section:2node} appears in $\pi_d : M_d\to \PP^1$ in many ways.  In order make this a bit more formal, let us first introduce a relevant notion.
 
 \begin{definition}[Equinodal arc]\label{def:}
 
A \emph{equinodal arc} in $\PP^1$ (with respect to the genus one fibration $\pi_d$) is  a smoothly  (embedded) arc $\g\subset  \PP^1$  whose relative interior  does not meet any critical value and over whose end points there are nodal fibers that define the same vanishing homology in a fiber over an interior point of $\g$.
\end{definition}
 
For such an equinodal arc $\g$ there exists a closed regular neighborhood $B_\g$ of $\g$  such that $\pi_d|B_\g$ is diffeomorphic to the genus one fibration investigated  in Section \ref{section:2node}.  
Corollary \ref{cor:MWpresentation} shows that then the part 
$\MW_\gamma(\pi_d)$ of the Mordell-Weil group $\MW(\pi_d)$ with support in $B_\g$ is infinite cyclic with a 
generator $F_\gamma$. It also tells us that there exists  a smooth oriented $2$-sphere $C_\g$ `suspended' over $\gamma$ with $C_\g\cdot C_\g=-2$ such that the associated Dehn twist $T(C_\g)\in \Mod(M_g)$ has a natural lift $\tau(C_\g)\in \Mod(\pi_d)$ with the latter is represented by  a diffeomorphism of $M_d$ over a diffeomorphism of $\PP^1$ that has its support in  $B_\g$  and is inside $B_\g$ the  basic (positive) braid that exchanges the endpoints of $\g$. They satisfy $\tau (C_\g)F_\g \tau (C_\g)^{-1}=F_\g^{-1}$. While  $\tau(C_\g)^2$  is nontrivial (for it acts on $\PP^1$ as a Dehn twist along $\p B_\g$), its image $T(C_\g)^2$ in $\Mod(M_d)$ is  trivial  
(even in the group of mapping classes  of $M_d$ that have their support over  $B_\g$).
 
 If $c_\g=[C_\g]\in \Lambda_d$, then $c_\g\cdot e=0$ and the above mapping classes act on $\Lambda_d$ as  an orthogonal reflection resp.\ by an Eichler transformation:
\begin{gather*}
\tau(C_\g)_*: x\mapsto  x+(c_\g\cdot x)c_\g\\
F_{\gamma *}=E(e\wedge c_\g): x\mapsto  x +(x\cdot e)c_\g-(x\cdot c_\g)e + (x\cdot e) e,
\end{gather*}
Let us refer to the image of $c_\g$ in $\Lambda_d(e)$, resp.\  $E(e\wedge c_\g$ in $\G_{d,e}$ as resp.\ an \emph{equinodal $(-2)$-vector}, 
an \emph{equinodal Eichler transformation} and to the reflection $\tau(C_\g)_*$  as  acting in $\Lambda_d(e)$ an \emph{equinodal reflection}.

\begin{theorem}\label{thm:equinodal}
These equinodal objects are maximal in the following sense:
\begin{enumerate}
\item[(i)] all the $(-2)$-vectors in $\Lambda_d(e)$ are equinodal and generate $\Lambda_d(e)$, or equivalently, the equinodal Eichler transformations generate  the  unipotent radical $\Lambda_d(e)$ of  $\G_{d,e}$,
\item[(ii)] the equinodal reflections  in $\Orth(\Lambda_d(e))^+$ make up a single conjugacy class and generate $\Orth(\Lambda_d(e))^+$.
\end{enumerate}
\end{theorem}

\begin{proof}
This theorem is known for $d=1$. In that case $\MW(\pi_1)$ is its holomorphic counterpart and isomorphic to $\Eb_8(-1)$ and $\Orth(\Lambda_1(e))^+$ is isomorphic to the Weyl group of the root system of $(-2)$-vectors in $\Eb_8(-1)$.   It is also known for $d=2$: we established this  in \cite{FL2} with  the help of the Torelli theorem for K3 surfaces. 

We now proceed with induction and assume $d>2$.  For any set $\Delta\subset \Lambda_d(e)$ of $(-2)$-vectors we denote  by $\G(\Delta)$ the subgroup of $\Orth(\Lambda_d(e))^+$ generated by the reflections in $\Delta$. Write $\Delta_d\subset \Lambda_d(e)$ for the set of equinodal $(-2)$-vectors.

We make use of the  fiber connected sum decomposition of $M_d$ obtained  via Theorem \ref{thm:fiberconnectedsum}. This  gives  a  decomposition of $\Lambda_d(e)$ 
\[
\Lambda_d(e)\cong \Eb_8(-1)\perp 2\Ub\perp \Eb_8(-1)\perp\cdots  \perp 2\Ub\perp \Eb_8(-1),
\]
where we have $d$ copies of $\Eb_8(-1)$ and  $d-1$ copies of $2\Ub$. This ordering of the summands should be understood as follows: if   $d>1$, then the copy  $H_{d-1}(e)_+$  resp.\  $H_{d-1}(e)_-$ of  $H_{d-1}(e)$ that supplements  the last three summands $2\Ub\perp \Eb_8(-1)$ resp.\ the first three summands $ \Eb_8(-1)\perp 2\Ub$ comes from the preimages over disks 
$D_\pm\subset \PP^1$ over whose boundary we have a trivial torus bundle. If we contract $\p D_\pm$ in $D_\pm$ and use a trivialization to lift that to a contraction of  $\pi_d^{-1}\p D_\pm$ in  $\pi_d^{-1}D_\pm$  to a torus, we get a fibration diffeomorphic  with  $\pi_{d-1}$.

The induction hypothesis then implies  that $\Delta_{\pm}:=\Delta_d\cap H_{d-1}(e)_\pm$ consists of  all the $(-2)$-vectors in
$H_{d-1}(e)_\pm$ and generates $H_{d-1}(e)_\pm$  and that  
$\G(\Delta_d)\supset \Orth(H_{d-1}(e)_+)^+\cup \Orth(H_{d-1}(e)_-)^+$. 
Note that $H_{d-1}(e)_+\cap H_{d-1}(e)_-$ is of type $H_{d-2}(e)$ and hence is generated by $(-2)$-vectors. This implies that  
$\Delta_d$ generates all of $\Lambda_d(e)$. It also shows that $\Delta_+\cup \Delta_-$ lies in a $\G(\Delta_+\cup \Delta_-)$-orbit. 
As $H_{d-1}(e)_+$ contains $\Eb_8(-1)\perp 2\Ub$ and hence a copy of $\Ab_2(-1)\perp 2\Ub$, it follows that the pair 
$(\Lambda_d(e), \Delta_+\cup \Delta_-)$ is a \emph{complete vanishing lattice} in the sense of Ebeling \cite{E}. His main theorem 2.3 (\emph{op.\ cit.}) states  that then $\G(\Delta_+\cup \Delta_-)=\Orth(\Lambda_d)^+$. So \emph{a fortiori},   $\G(\Delta_d)=\Orth(\Lambda_d)^+$.
It is well-known that the $(-2)$-vectors in a lattice of type $\Lambda_d$ make up a single $\Orth(\Lambda_d)^+$-orbit; this  also  follows from Ebelings' s proposition  (2.5). 
\end{proof}

\begin{corollary}\label{cor:repisonto}
For each $d\geq 1$, the representation of  $\Mod(\pi_d)$ on $\Lambda_d$ has image $\G_{d,e}$. 
\end{corollary}
\begin{proof}
Proposition \ref{prop:MWcohomologyII} asserts that the subgroup $\MW(\pi_d)$ maps (in fact, isomorphically) onto the unipotent radical ($\cong \Lambda_d(e)$ of $\G_{d,e}$ and  
Theorem \ref{thm:equinodal} tells us  $\Mod(\pi_d)$  maps onto $\Orth(\Lambda_d(e))^+$. Hence  $\MW(\pi_d)$ maps onto $\G_{d,e}$.
\end{proof}

\begin{remark}\label{rem:repnonfaithful}
The representation of $\Mod(\pi_d)$ on $\Lambda_d$ is not faithful:
in Corollary  \ref{cor:MWpresentation} we found a  $(-2)$-sphere $C$  defining  a relative mapping class $\tau(C)\in \Mod(\pi_d)$ of infinite order whose image in $\Mod (M_d)$ is the 2-dimensional Dehn twist $T(C)$, which  has  order $2$ (and which induces an orthogonal reflection in $\Lambda_d$).  This also shows that the forgetful  homomorphism $\Mod(\pi_d)\to \Mod (M_d)$ is not injective.  
\end{remark}

\begin{proof}[Proof of  Corollary \ref{corollary:small}]
\label{proof:Nielsen1}
We first show that the lattice $\Lambda_d(e)$ admits a basis $\Cscr$ of $(-2)$-vectors. Since $\Lambda_d(e)$ is the isomorphic to the orthogonal direct sum of $\Eb_8(-1)$ and $d-1$ copies of $\Eb_8(-1)\perp 2(-1)\Ub$, it suffices to prove this for these two lattices. This is clear for  
$\Eb_8(-1)$, for any root basis $\alpha_1, \dots, \alpha_8$ will do. If $e,f$ resp.\  $e',f'$ are standard (isotropic) bases for the first resp.\ second $\Ub$-summand, then $\{\alpha_1, \dots, \alpha_8, \alpha+e, \alpha+f, \alpha+e', \alpha+f'\}$ is basis of $\Eb_8(-1)\perp 2(-1)\Ub$ consisting of $(-2)$-vectors.

By Theorem \ref{thm:equinodal}  every $c\in \Cscr$ is an equinodal $(-2)$-vector, so associate  to some equinodal arc 
$\g_c$. By definition there exists a translation   $f_c\Trans(\pi_d)$ with  support contained in a regular neighborhood of $\g$ such that $F_\g$ induces the Eichler transformation $E(c\wedge e)$. Then the homomorphism $\Lambda_d(e)\to \Trans(\pi_d)$ of abelian groups that takes $c$ to  $f_c$ gives the desired  Nielsen realization.
\end{proof}

\begin{proof}[Proof of Theorem \ref{thm:smallsupport}]
Given $c$ as in the hypothesis of the theorem, Theorem  
\ref{thm:equinodal} produces an equinodal arc $\gamma_c\subset\Pb^1$ and its tubular neighborhood $U_c$, giving part (1).   Part (2a) of the theorem 
is part of the statement of Proposition \ref{prop:spheres and sections}.  Part (2b) is  Corollary \ref{cor:2Dehn}.
\end{proof}

\begin{remark}\label{rem:}  For $h=(h_M,h_{\PP^1}) \in \Diff(\pi_d)\subset \Diff(M_d)\times \Diff(\PP^1)$,  the  mapping torus of $h_{\PP^1}$  determines a spherical braid. Its  
$\Diff^+(\PP^1)$-conjugacy class only depends on the image of $h$ in $\Mod(\pi_d)$ and can be understood as an element of the orbifold fundamental group of $\Scal_{12d}\bs\Mcal_{12d}$ (the moduli space of $12d$-element subsets of $\PP^1$ given up to projective equivalence). We thus have defined a homomorphism
\begin{equation}
\label{eq:sbraid}
\Bscr: \Mod(\pi_d)\to \pi_1^\orb(\Scal_{12d}\bs\Mcal_{12d},[D]), 
\end{equation}
where $D$ is the discriminant of $\pi_d$.  The right-hand side of \eqref{eq:sbraid} is a quotient of the mapping class group of the pair $(\PP^1,D)$ by  $\pi_1(\SO_3)\cong\{\pm1 \}$). The kernel of $\Bscr$ is $\Mod(M_d/\PP^1)$, the connected component group of 
$\Diff(M_d/\PP^1)$, which contains $\MW(\pi_d)$ as a subgroup 
of index two (see Remark \ref{rem:index2}). 

The image of $\Bscr$ 
certainly contains the simple spherical braid  around a regular neighborhood boundary of an equinodal arc.
It also contains the third power of the spherical braids defined by the regular neighborhood boundary of what we might 
call an an \emph{anti-equinodal arc} $\g$: such an arc connects two points of the discriminant that define on a fiber over its 
interior vanishing cycles that span its first homology of that fiber as in Proposition  \ref{lemma:less interesting case} (if we let $\g$ shrink to a singleton, then its preimage becomes a Kodaira fiber of cuspidal type and hence the monodromy over a regular  neighborhood boundary has order $6$). In either case, the squares of these elements become trivial in $\Mod(M_d)$.
\end{remark}

\begin{question}\label{quest:}
Is the image of $\Bscr$ generated  by  the simple spherical braids associated to the equinodal arcs and the third  power  of the simple spherical braids associated to the anti-equinodal arcs? Is the subgroup generated by their squares equal to the kernel of the natural map $\Mod(\pi_d)\to \Mod (M_d)$? 
\end{question}

\end{document}